\documentclass[a4paper, reqno, 11pt]{amsart}
\usepackage[utf8]{inputenc}

\usepackage{amsthm}
\usepackage{amssymb}
\usepackage[american]{babel}
\usepackage{exscale}
\usepackage[mathscr]{euscript}
\usepackage{url}
\usepackage[all,ps,tips,tpic]{xy}
\usepackage{graphicx}
\usepackage{color}

\newtheorem{lem}{Lemma}[section]

\newtheorem{definition}[lem]{Definition}
\newtheorem{defprop}[lem]{Definition/Proposition}

\newtheorem{cor}[lem]{Corollary}
\newtheorem{thm}[lem]{Theorem}
\newtheorem{prop}[lem]{Proposition}

\newtheorem{conjecture}[lem]{Conjecture}
\newtheorem{question}[lem]{Question}

\theoremstyle{remark}
\newtheorem{rem}[lem]{Remark}

\DeclareMathOperator{\Hom}{Hom}

\newcommand{\F}{\mathbb{F}}

\newcommand{\tr}{\operatorname*{tr}}
\newcommand{\GL}{\mathrm{GL}}
\newcommand{\SL}{\mathrm{SL}}

\def\tilde{\widetilde}

\newcommand{\bG}{\mathbf{G}} 
\newcommand{\bH}{\mathbf{H}} 
\newcommand{\bT}{\mathbf{T}} 
\newcommand{\bh}{\mathbf{h}} 
\newcommand{\Sh}{\mathrm{Sh}} 

\newcommand{\triv}{{\mathbf{1}}} 
\newcommand{\cH}{\mathcal{H}} 
\newcommand{\cF}{\mathcal{F}}

\newcommand{\cZ}{\mathcal{Z}}
\newcommand{\cY}{\mathcal{Y}}
\newcommand{\cL}{\mathcal{L}}
\newcommand{\el}{\mathrm{ell}}
\newcommand{\st}{\mathrm{st}}
\newcommand{\vol}{\mathrm{vol}}
\newcommand{\KT}{\mathrm{KT}} 
\newcommand{\Res}{\mathrm{Res}} 
\newcommand{\GHreg}{(\bG,\bH)\mbox{-}\mathrm{reg}}
\newcommand{\semis}{\mathrm{ss}} 
\newcommand{\disc}{\mathrm{disc}} 
\newcommand{\der}{\mathrm{der}} 
\newcommand{\ur}{\mathrm{ur}} 
\newcommand{\Z}{\mathbb{Z}} 
\newcommand{\R}{\mathbb{R}} 
\newcommand{\Q}{\mathbb{Q}} 
\newcommand{\C}{\mathbb{C}} 
\newcommand{\A}{\mathbb{A}} 
\newcommand{\G}{\mathbb{G}} 
\newcommand{\cE}{\mathcal{E}} 
\newcommand{\cO}{\mathcal{O}} 
\newcommand{\cK}{\mathcal{K}} 

\newcommand{\ra}{\rightarrow} 
\newcommand{\hra}{\hookrightarrow} 
\newcommand{\Gal}{\mathrm{Gal}} 
\newcommand{\ord}{\mathrm{ord}} 
\newcommand{\nind}{{\rm n\textrm{-}ind}} 
\newcommand{\Spl}{\mathrm{Spl}}
\newcommand{\bs}{\backslash} %
\newcommand{\fkp}{{\mathfrak p}}
\newcommand{\fkS}{{\mathfrak S}}
\newcommand{\fkK}{{\mathfrak K}}
\def\hat{\widehat}  
\def\lg{\langle}
\def\rg{\rangle}
\def\ol{\overline}

\newdir{ >}{{}*!/-5pt/@{>}}
\SelectTips{cm}{10}

\begin{document}
\title[On the cohomology of Shimura varieties]{On the cohomology of compact unitary group Shimura varieties at ramified split places}
\author{Peter Scholze, Sug Woo Shin}
\begin{abstract}
In this article, we prove results about the cohomology of compact unitary group Shimura varieties at split places. In nonendoscopic cases, we are able to give a full description of the cohomology, after restricting to integral Hecke operators at $p$ on the automorphic side. We allow arbitrary ramification at $p$; even the PEL data may be ramified. This gives a description of the semisimple local Hasse-Weil zeta function in these cases.

We also treat cases of nontrivial endoscopy. For this purpose, we give a general stabilization of the expression given in \cite{ScholzeDefSpaces}, following the stabilization given by Kottwitz in \cite{KottwitzStabilization}. This introduces endoscopic transfers of the functions $\phi_{\tau,h}$ introduced in \cite{ScholzeDefSpaces}. We state a general conjecture relating these endoscopic transfers with Langlands parameters.

We verify this conjecture in all cases of EL type, and deduce new results about the endoscopic part of the cohomology of Shimura varieties. This allows us to simplify the construction of Galois representations attached to conjugate self-dual regular algebraic cuspidal automorphic representations of $\GL_n$, as previously constructed by one of us, \cite{Shin11}.
\end{abstract}

\setcounter{tocdepth}{1}

\date{\today}
\maketitle
\tableofcontents
\pagebreak

\section{Introduction}

This paper deals with the problem of determining the Galois action on the cohomology of Shimura varieties, specifically at places of bad reduction. Let us first briefly recall the expected description, due to Langlands.

Let $\Sh_K$ be some Shimura variety associated to a reductive group $\bG$ over $\Q$ and a compact open subgroup $K\subset \bG(\A_f)$, and some additional data. It is defined over a number field $E$, canonically embedded into $\C$. For definiteness, let us assume that the Shimura variety is compact, as we will only deal with this case; it is equivalent to requiring that $\bG$ is anisotropic modulo center\footnote{In the general case, one considers the intersection cohomology groups of the Baily-Borel compactification, and similar results are expected.
cf. \cite{Mor10}.}. Then one considers the etale cohomology groups
\[
H^i = \varinjlim_K H^i_{\mathrm{et}}(\mathrm{Sh}_K\otimes_E \bar{\Q},\bar{\mathbb{Q}}_\ell)\ ,
\]
which carry an action of $\bG(\A_f)$ (via Hecke correspondences) and of $\Gal(\bar{\Q}/E)$. In the main body of the paper we consider the cohomology groups with coefficients in a local system associated to an algebraic representation $\xi$ of $\bG$, where everything works without essential change. Let $H^\ast$ be the alternating sum of the $H^i$ in a suitable Grothendieck group; then one can write
\[
H^\ast = \sum_{\pi_f} \pi_f\otimes \sigma(\pi_f)\ ,
\]
where $\pi_f$ runs through irreducible admissible representations of $\bG(\A_f)$, and $\sigma(\pi_f)$ is some virtual finite-dimensional representation of $\Gal(\bar{\Q}/E)$.

If one forgets about the action of $\Gal(\bar{\Q}/E)$ and is only interested in describing the $\bG(\A_f)$-action, then one can base-change to $\C$ and use Matsushima's formula, which computes the cohomology as the real-analytic de Rham cohomology of $\Sh_K(\C)$. As $\Sh_K(\C)$ is a (finite disjoint union of) locally symmetric varieties for $\bG(\R)$, automorphic forms for $\bG$ enter the stage, and in particular one sees that if $\pi_f$ appears in $H^\ast$, then there is some $\pi_\infty$ such that $\pi_f\otimes \pi_\infty$ is an irreducible automorphic representation of $\bG(\A)$. Using relative Lie algebra cohomology of $\pi_\infty$, one expresses the contribution of any automorphic representation $\pi_f\otimes \pi_\infty$ of $\bG(\A)$ to $H^\ast$.

The situation becomes more complicated when one is trying to understand the Galois action. Let us recall the general recipe given in Kottwitz' paper \cite{KottwitzStabilization}. Let us restrict to the contribution of tempered representations for simplicity. (The general case is similar, using $A$-parameters instead.) Let us remark that tempered representations should only show up in the middle degree cohomology, so that this really only gives part of the whole picture.

First, note that any irreducible automorphic representation $\pi$ of $\bG(\A)$ should give rise to a global Langlands parameter, which should be some map
\[
\varphi_{\pi}: \mathcal{L}_\Q\ra {}^L \bG\ ,
\]
where $\mathcal{L}_\Q$ is the conjectural global Langlands group, and ${}^L \bG$ is the $L$-group of $\bG$, which is the semidirect product of the dual group $\hat{\bG}$ over $\C$ and $\Gal(\bar{\Q}/\Q)$. The parameter $\varphi_{\pi}$ should be discrete due to our assumption on $\bG$.
Fix an isomorphism $\bar{\mathbb{Q}}_\ell\cong \C$. Henceforth the subscript $\ell$ for a complex group will designate the base change from $\C$ to $\ol{\Q}_l$. In the cases of interest here, such a parameter $\varphi_{\pi}$ should be closely related to an $\ell$-adic Galois representation
\[
\varphi_{\pi,\ell}: \Gal(\bar{\Q}/\Q)\ra {}^L \bG_\ell \ .
\]
\footnote{We refer to \cite{BuzzardGee} for a detailed discussion of this point. There is some subtle twisting issue which we ignore here.} A discrete tempered global Langlands parameter
\[
\varphi: \mathcal{L}_\Q\ra {}^L \bG\ ,
\]
upon restriction to each place $v$ of $\Q$, should give rise to local $L$-packets $\Pi(\varphi_v)$, finite sets of isomorphism classes of representations of $\bG(\Q_v)$. We let $\Pi_f(\varphi)$ be the set of isomorphism classes of irreducible admissible representations $\pi_f$ of $\bG(\A_f)$ such that $\pi_p\in \Pi(\varphi_p)$ for all $p$.

On the other hand, the data defining the Shimura variety give rise to a natural representation
\[
r: {}^L (\bG\otimes_{\Q} E)\ra \GL_N\ ,
\]
for some integer $N$. In particular, one gets an $\ell$-adic Galois representation $r\circ \varphi_{\ell}|_{\Gal(\bar{\Q}/E)}$:
\[
\Gal(\bar{\Q}/E)\buildrel {\varphi_{\ell}}\over\ra {}^L (\bG\otimes_{\Q} E)_\ell\buildrel r\over\ra \GL_N(\bar{\Q}_\ell)\ .
\]
A first approximation would be that the contribution of tempered representations to $H^\ast$ is given by (up to twists)
\begin{equation}\label{e:cohom-Sh-naive}
\sum_{\varphi} \sum_{\pi_f\in \Pi_f(\varphi)} m(\pi_f) \pi_f\otimes (r\circ \varphi_{\ell}|_{\Gal(\bar{\Q}/E)})\ .
\end{equation}
Here $m(\pi_f)$ are certain (possibly negative) integers related to multiplicities of automorphic representations.

This approximation turns out to be correct in the situations where one can ignore endoscopy. A part of this paper will be restricted to such situations, but the construction of Galois representations requires working in a more general context.

The additional ingredient needed is the group $S_\varphi\subset \hat{\bG}$ of self-equivalences of $\varphi$, i.e. the subgroup of $\hat{\bG}$ which centralizes the image of $\varphi$ up to a locally trivial 1-cocycle of $\cL_{\Q}$ into $Z(\hat{\bG})$, cf. \cite{Kot84a}, \S11. Let $\mathfrak{S}_\varphi = S_\varphi / Z(\hat{\bG})$, which is a finite abelian group in the cases of interest\footnote{In fact, in most cases it is just a product of some copies of $\mathbb{Z}/2\mathbb{Z}$.}. If $\mathfrak{S}_\varphi$ is trivial, then the contribution of $\varphi$ to \eqref{e:cohom-Sh-naive} should be correct. Such $\varphi$ are called stable.

To give the general recipe, let us look at the corresponding $\ell$-adic version. We get an action of $S_{\varphi,\ell}$ on the Galois representation $r\circ \varphi_{\ell}|_{\Gal(\bar{\Q}/E)}$. One checks that the subgroup $Z(\hat{\bG})_\ell$ acts by a certain fixed character $\mu_1$. Let $\nu$ be a character of $S_{\varphi,\ell}$ whose restriction to $Z(\hat{\bG})_\ell$ is $\mu_1$, and denote by $(r\circ \varphi_{\ell}|_{\Gal(\bar{\Q}/E)})_\nu$ the direct summand of $r\circ \varphi_{\ell}|_{\Gal(\bar{\Q}/E)}$ on which $S_{\varphi,\ell}$ acts through the character $\nu$. Note that the set of such $\nu$ is a principal homogeneous space under the dual group of $\mathfrak{S}_\varphi$.

Now, the general contribution of tempered representations to $H^\ast$ should be (up to twists)
\begin{equation}\label{e:cohom-Sh}
\sum_\varphi \sum_{\pi_f\in \Pi_f(\varphi)} \sum_{\nu} m(\pi_f,\nu) \pi_f\otimes (r\circ \varphi_{\ell}|_{\Gal(\bar{\Q}/E)})_\nu\ .
\end{equation}
Again, $m(\pi_f,\nu)$ are certain integers, whose general definition is somewhat subtle.

The object of this paper is to verify such a description in as many cases as possible. As even the ingredients in this formula remain largely conjectural, one has to be more specific about the precise statement one wants to prove. We fix a prime $\fkp$ of $E$ above a rational prime $p\neq \ell$, and restrict the representation of $\Gal(\bar{\Q}/E)$ to the Weil group $W_{E_\fkp}$ of $E_\fkp$. Then $\pi_p$ has a local $L$-parameter, which gives a map
\[
\varphi_{\pi_p}: W_{\Q_p}\ra {}^L \bG_\ell\ .
\]
One expects that if $\pi_f\in \Pi(\varphi)$, then $\varphi_{\pi_p}$ is the restriction of $\varphi_\ell$ to the local Weil group; this is known as the local-global compatibility. In particular, we have
\[
(r\circ \varphi_{\ell}|_{\Gal(\bar{\Q}/E)})|_{W_{E_\fkp}} = r\circ \varphi_{\pi_p}|_{W_{E_\fkp}}\ ,
\]
so that the local Weil group representations are determined by the local components of the representation $\pi_f$. We want to prove that the contribution of tempered representations to $H^\ast$ is
\begin{equation}\label{e:cohom-Sh-local}
\sum_\varphi \sum_{\pi_f\in \Pi_f(\varphi)} \sum_{\nu} m(\pi_f,\nu) \pi_f\otimes (r\circ \varphi_{\pi_p}|_{W_{E_\fkp}})_\nu
\end{equation}
as a virtual $\bG(\A_f)\times W_{E_\fkp}$-representation.

This still requires the existence of the local Langlands correspondence for the group $\bG(\Q_p)$. By \cite{HT01} and \cite{Hen00}, it is known for general linear groups. It is in this setup that we prove our results. In other words, we start with a Shimura variety associated to a reductive group $\bG$ over $\Q$ which is anisotropic modulo center, such that locally at $p$, $\bG$ is a product of Weil restrictions of general linear groups. As $\bG$ will be a unitary (similitude) group in this case, we use the term compact unitary group Shimura varieties, and the prime $p$ is called a split place.\footnote{A slightly unfortunate terminology, as it only implies that $\bG_{\Q_p}$ is quasisplit, but not that it is split. A justification is that the primes above $p$ should be split between the totally real field and its CM quadratic extension defining the hermitian form for $\bG$.} We note that we allow both the representation $\pi_p$ and the implicit CM field defining the unitary group $\bG$ to be ramified at $p$.

Our first theorem is the following. Here, we allow coefficients given by a general irreducible algebraic representation $\xi$ of $\bG$. The corresponding cohomology group is denoted $H^\ast_\xi$.

\begin{thm} Assume that $\Sh$ is a compact unitary group Shimura variety, $p$ is a split place, and endoscopy is trivial, as in \cite{KottwitzLambdaAdic}. Then
\[
H^\ast_\xi = \sum_{\pi_f} a(\pi_f) \pi_f\otimes (r\circ \varphi_{\pi_p}|_{W_{E_\fkp}})|\cdot|^{-\dim \Sh/2}
\]
as a virtual $\bG(\A_f^p)\times \bG(\Z_p)\times W_{E_\fkp}$-representation. Here $\pi_f$ runs through irreducible admissible representations of $\bG(\A_f)$, the integer $a(\pi_f)$ is as in \cite{KottwitzLambdaAdic}, p. 657, and $\varphi_{\pi_p}$ is the local Langlands parameter associated to $\pi_p$. Also $\bG(\Z_p)$ is a certain maximal compact subgroup of $\bG(\Q_p)$.
\end{thm}

The restriction of the $\bG(\Q_p)$-action to $\bG(\Z_p)$ is due to our method of proof. Let us remark that in the situation of the theorem, one expects that the $p$-component $\pi_p$ is determined by $\pi_f^p$. Granting this, the identity would extend to an identity of $\bG(\A_f)\times W_{E_\fkp}$-representations. Unconditionally, the theorem gives a description of the semisimple Hasse-Weil zeta function.

\begin{cor} In the situation of the theorem, let $K\subset \bG(\A_f)$ be any sufficiently small compact open subgroup. Then the semisimple local Hasse-Weil zeta function of $\Sh_K$ at the place $\fkp|p$ of $E$ is given by
\[
\zeta_\fkp^\semis (\Sh_K,s) = \prod_{\pi_f} L^\semis(s-\dim \Sh/2, \pi_p, r)^{a(\pi_f)\dim \pi_f^K}\ .
\]
\end{cor}

We remark that in these nonendoscopic cases, no form of the fundamental lemma or related work is needed, except for the existence of base change transfers for $\GL_n$, as proved by Arthur-Clozel, \cite{ArthurClozel}.

The second theorem is a description of the cohomology of compact
  unitary group Shimura varieties at split primes in endoscopic cases.
  In this case, the full machinery of endoscopy (for unitary groups) is used, and we have to make some additional assumptions
  concerning our data to establish unconditional results.
  Specifically, consider a unitary group $\bG$
  relative to a CM quadratic extension $F/F_0$ such that $F$ contains
  an imaginary quadratic field $\cK$. We assume that $F_0\neq \Q$.
  Fix a rational prime $p$ split in $\cK$ and a prime $\fkp$ of
  the reflex field $E$ above $p$. We assume that $\bG$ is quasisplit locally at all finite places but we do not fix the signature at infinity,
  except that $\bG$ has a compact factor at one infinite place.
  The group $\bG(\A_\cK)$ admits an order 2 automorphism $\theta$ induced by
  the complex conjugation on $\cK$.
  Let $\Pi$ be a $\theta$-stable automorphic representation of $\bG(\A_\cK)\cong \GL_1(\A_\cK)\times \GL_n(\A_F)$
  which has infinite component determined by $\xi$. We define
  a virtual representation of $\Gal(\ol{\Q}/E)$
  \begin{equation}\label{e:H(pi_f)}
    H^*_\xi(\pi_f)=\Hom_{\bG(\A_f)} (\pi_f,H^*_\xi)
  \end{equation}
 and also the $\Pi$-part $H^*_\xi(\Pi)$ of $H^*_\xi$.
 The latter is roughly the sum of $ H^*_\xi(\pi_f)$ over the set of $\pi_f$ whose
 base change is $\Pi_f$ (cf. \eqref{e:Pi-part}).
  We consider two cases, which are loosely described as:
\begin{itemize}
  \item (Case ST) $\Pi$ is cuspidal.
\item (Case END) $\Pi$ is induced from a $\theta$-stable cuspidal automorphic representation on
a maximal proper Levi subgroup of $\bG(\A_\cK)$.
\end{itemize}

\begin{thm}\label{MainTheorem4} The following hold up to an explicit constant (cf. \eqref{e:C}). In (Case ST)
\[
H^*_\xi(\Pi)|_{W_{E_{\fkp}}}=
(r\circ \varphi_{\pi_p}|_{W_{E_{\fkp}}})|\cdot|^{-\dim\Sh/2}\ .
\]
In (Case END) there exists a natural decomposition
\[
r\circ \varphi_{\pi_p}|_{W_{E_{\fkp}}} = (r\circ \varphi_{\pi_p}|_{W_{E_{\fkp}}})_1 + (r\circ \varphi_{\pi_p}|_{W_{E_{\fkp}}})_2
\]
into two representations, and some $i\in \{1,2\}$ depending only on $\Pi$ such that
\[
H^*_\xi(\Pi)|_{W_{E_{\fkp}}}= (r\circ \varphi_{\pi_p}|_{W_{E_{\fkp}}})_i |\cdot|^{-\dim\Sh/2}\ .
\]
\end{thm}

As a corollary of this theorem, one can reprove the existence of Galois representations associated to regular algebraic conjugate selfdual cuspidal automorphic representations of $\GL_n$ over CM fields, as previously constructed by one of us, \cite{Shin11}, cf. Theorem \ref{t:ExGalRepr}. We can also reprove the Ramanujan-Petersson conjecture at unramified places, cf. \cite{ClozelPurity}:

\begin{cor} Let $F$ be a CM field, and let $\Pi$ be a regular algebraic conjugate selfdual cuspidal automorphic representation for $\GL_n(\A_F)$. Then for all finite places $v$ of $F$ lying above a rational prime $p$ above which $F$ and $\Pi$ are unramified, the component $\Pi_v$ is tempered.
\end{cor}

Although our theorem gives information in the ramified case as well, we do not see a direct way of proving temperedness at ramified places, as done by Caraiani, \cite{Caraiani}.

Let us explain the strategy of proof. This paper is a sequel to \cite{ScholzeDefSpaces}. We recall that under certain circumstances, the main theorem of \cite{ScholzeDefSpaces} gives an equality of the form
\begin{equation}\label{e:MT1}
\tr(\tau\times hf^p | H^\ast_\xi) = \sum_{\substack{(\gamma_0;\gamma,\delta)\\ \alpha(\gamma_0;\gamma,\delta)=1}} c(\gamma_0;\gamma,\delta) O_\gamma(f^p)TO_{\delta\sigma}(\phi_{\tau,h}) \tr \xi(\gamma_0)
\end{equation}
for the trace of a Hecke operator $hf^p$ twisted with an element of the local Weil group $\tau$ at a prime $\fkp|p$ of the reflex field $E$. We refer to Section \ref{s:proof-Thm1.2}, especially Theorem \ref{MainTheorem1}, for a precise statement. Here, $h\in C_c^\infty(\bG(\Z_p))$ for a certain extension of $\bG$ to a group scheme over $\Z_{(p)}$. There is no naive generalization of this theorem to general $h$, and for this reason our method only gives information about the $\bG(\A_f^p)\times \bG(\Z_p)\times W_{E_\fkp}$-action.

We give a general stabilization of this expression, following the arguments of Kottwitz in \cite{KottwitzStabilization}.
Write $\cE_{\el}(\bG)$ for the set of isomorphism classes of elliptic endoscopic triples for $\bG$.
For $(\bH,s,\eta)\in \cE_{\el}(\bG)$, let $f^{\bH}=f^{\bH,p}f^{\bH}_{\tau,h} f^{\bH}_\xi$
be the function on $\bH(\A)$ obtained from $f^p$, $\phi_{\tau,h}$ and $\xi$.
Here $f^{\bH,p}$ is the usual transfer of $f^p$, whereas
$f^{\bH}_{\tau,h}$ is a twisted transfer of $\phi_{\tau,h}$ relative to a suitable $L$-morphism.
The function $ f^{\bH}_\xi$ at infinity is an explicit linear combination
of pseudo-coefficients for discrete series.
Let $ST^{\bH}_{\el}$ denote the stable distribution on $\bH(\A)$ given as the sum of stable orbital integrals on
elliptic semisimple elements of $\bH(\Q)$.

\begin{thm}\label{MainTheorem2}
The formula \eqref{e:MT1} is stabilized as follows, where the sum runs over  $\cE_{\el}(\bG)$.
$$\tr(\tau\times h f^p|H^{\ast}_{\xi}) = \sum_{(\bH,s,\eta)} \iota(\bG,\bH) ST^{\bH}_{\el} (f^{\bH})$$
\end{thm}

  To apply Theorem \ref{MainTheorem2}, it is essential to understand character identities satisfied by $f^{\bH}_{\tau,h}$.
  We are able to formulate a precise conjecture, cf. Conjecture \ref{c:char-id} and Conjecture \ref{c:char-id-st-Bernstein},
  which is purely local in nature and may be thought of as a common
  generalization of the identities \cite[Thm 1.2.(a)]{ScholzeLLC} and \cite[(9.7)]{KottwitzStabilization}.
  As explained in the introduction of \cite{ScholzeDefSpaces}, it is also related to a conjecture of Kottwitz
  about similar functions arising via the bad reduction of Shimura varieties of parahoric type, and related
  ideas were expressed recently by Haines.

 The conjecture relates (the endoscopic transfers of) the functions $\phi_{\tau,h}$ arising from deformation spaces of
 $p$-divisible groups with the local Langlands Correspondence. In this sense it is similar in spirit to the conjectures
 that describe the cohomology of Rapoport-Zink spaces in terms of the local Langlands Correspondence; the precise
 relation between the two points of view remains to be clarified. Let us only mention here that our conjecture has the
 advantage that all Langlands parameters of $\bH$ occur in the same way, as was the case in \cite{ScholzeLLC}, Theorem 1.2.

  As evidence we verify the conjecture when all data are unramified, cf. Lemma \ref{l:conj-unramified}.
  More importantly, we prove the conjecture in all cases of EL type, cf. Corollary \ref{c:conj-GLn}.
  This extends the result of \cite{ScholzeLLC}, Theorem 1.2, to arbitrary signature. We stress here that we allow ramified EL data.

  Granting this conjecture, we follow the arguments of \cite{KottwitzStabilization} and verify that for any compact PEL type Shimura variety
  of type A or C one gets the expected description of the cohomology, assuming
  some additional hypotheses on the global $A$-packet classification and the stable trace formula, as in \cite{KottwitzStabilization}.
  This justifies the general form of the conjecture.

  Knowing the general character identities satisfied by (the endoscopic transfers of) the test function $\phi_{\tau,h}$, our main theorems
  follow from Theorem \ref{MainTheorem2} and applications of suitable trace formulas.

  We remark on one more application of our arguments in the nonendoscopic case. Namely, we prove that
  for certain compact unitary group Shimura varieties without endoscopy, all Newton strata are nonempty at "split" places in the above sense, cf.
  Corollary \ref{cor:AllPDivGroupsOccur} and Corollary \ref{cor:NewtonNonempty}. In particular, this proves that the
  deformation spaces considered in \cite{ScholzeDefSpaces} are algebraizable in all cases of EL type, cf. Corollary \ref{cor:ELAlg}.

  Let us summarize the content of the different sections. In the first sections, we give the stabilization leading to Theorem \ref{MainTheorem2}. We also explain the process of pseudostabilization in the nonendoscopic case, cf. Section \ref{s:pseudostab}. Afterwards, we state the conjecture about the character identitites satisfied by the endoscopic transfers $f_{\tau,h}^\bH$ of $\phi_{\tau,h}$ in Section \ref{s:char-id}. This will involve the theory of the stable Bernstein center, which we recall in Section \ref{s:stBernstein}. We go on by proving the conjecture in cases of EL type in Section \ref{s:prove-conj}. The last sections apply this result to prove our main theorems. We note that a reader only interested in our results in the nonendoscopic case finds the argument in Section \ref{s:pseudostab}, the first half of Section \ref{s:prove-conj}, and Section \ref{s:lgc}.

{\bf Acknowledgments.} This work was started as a result of a discussion at the IAS in Princeton during the special year on Galois representations and automorphic forms. The authors want to thank this institution for the hospitality, and also thank T. Haines for related discussions. The first-named author wants to thank R. Taylor for the invitation to speak there. Moreover, he thanks M. Rapoport for his constant encouragement and help. The second-named author wishes to thank D. Kazhdan and D. Vogan for answering his questions. This work was written while the first-named author was a Clay Research Fellow. The second-named author's work was supported by the National Science Foundation during his stay at the Institute for Advanced Study under agreement No. DMS-0635607.\footnote{Any opinions, findings and conclusions or recommendations expressed in this material are those of the author(s) and do not necessarily reflect the views of the National Science Foundation.}

\section{Kottwitz triples}\label{s:Kottwitz-triples}

  The definitions of this section are group-theoretic
  and do not require any reference to Shimura varieties even though
  that is where the motivation comes from.
  Let $\bG$ be a connected reductive group over $\Q$ with $A_{\bG}$ the maximal split component of its center such that
\begin{enumerate}
\item $(\bG/A_{\bG})(\R)$ contains a compact maximal torus,
\item $\bG^{\der}$ is simply connected,
\item $\bG_{\Q_p}$ is quasisplit.
\end{enumerate}
  Condition (1) is indispensable. Both (2) and (3) are not essential, but are assumed to avoid complication.

  Let $\bh:\Res_{\C/\R}\G_m \ra \bG_\R$ be a group morphism. Taking the basechange to $\C$, this gives rise to a morphism $\G_m\times \G_m\ra \bG_\C$, and we denote the first factor by $\mu: \G_m\ra \bG_\C$. Choosing a maximal torus $\bT$ of $\bG$ over $\C$, one can replace $\mu$ by a conjugate which factors over $\bT$; this induces by duality a cocharacter of $\hat{\bT}$. Restricting this cocharacter to $Z(\hat{\bG})$ gives an element $\mu_1\in X^\ast(Z(\hat{\bG}))$ independent of all choices\footnote{In particular, we revert back to the normalization used by Kottwitz, which is different from the one used in \cite{ScholzeDefSpaces}.}.

 The $\bG(\C)$-conjugacy class $\ol{\mu}$ of $\mu$ has its field of definition $E\subset \C$, which is finite over $\Q$. We fix a prime $\fkp$ of $E$ above $p$.
  Let $f(\fkp)$ denote the inertia degree of $E_\fkp$ over $\Q_p$.

  Set $L$  to be the fraction field of $W(\ol{\F}_p)$. Write $B(\bG_{\Q_p})$ for the set of $\sigma$-conjugacy classes in $\bG(L)$.
  Let
  \[
  \kappa_{\bG,p}:B(\bG_{\Q_p})\ra X^*(Z(\hat{\bG})^{\Gamma(p)})
  \]
  denote the Kottwitz map, cf. \cite{KottwitzStabilization}, Section 6.
  Here $\Gamma(p)$ denotes the absolute Galois group of $\Q_p$.
  Use $\bG(\Q)_{\semis}$ to denote the set of semisimple elements in $\bG(\Q)$.

\begin{definition}Let $j\geq 1$ be an integer, and set $r=j f(\fkp)$.
  A degree $j$ Kottwitz triple $(\gamma_0;\gamma,\delta)$ consists of
  $\gamma_0\in \bG(\Q)_{\semis}$, $\gamma\in \bG(\A^p_f)$ and $\delta\in \bG(\Q_{p^{r}})$ such that
\begin{itemize}
  \item $\gamma_0$ is elliptic in $\bG(\R)$,
  \item $\gamma_0$ and $\gamma$ are stably conjugate in $\bG(\Q_v)$ for all $v\neq p,\infty$,
  \item $\gamma_0$ and $N\delta=\delta \delta^{\sigma}\cdots \delta^{\sigma^{r-1}}$
  are stably conjugate, i.e. conjugate in $\bG(\ol{\Q}_p)$,
  \item $\kappa_{\bG,p}(\delta)=-\mu_1$ in $X^*(Z(\hat{\bG})^{\Gamma(p)})$.
\end{itemize}
  Two triples $(\gamma_0;\gamma,\delta)$ and $(\gamma'_0;\gamma',\delta')$ are considered equivalent
  if $\gamma_0\sim_{\st} \gamma'_0$, $\delta\sim_{\sigma} \delta'$ and
  $\gamma_v\sim \gamma'_v$ for all $v\neq p,\infty$.
  Denote by $\KT_j$ (a set of representatives for) equivalence classes of degree $j$ Kottwitz triples.

\end{definition}

  Let $I_0$ be the centralizer of $\gamma_0$ in $\bG$. Kottwitz defines
  an invariant $\beta_v(\gamma_0;\gamma,\delta)\in X^*(Z(\hat{I}_0)^{\Gamma(v)}Z(\hat{\bG}))$
  for each place $v$ of $\Q$ in \cite{KottwitzStabilization}, Section 2.
  Set $\tilde{\fkK}(I_0/\Q)=\cap_v Z(\hat{I}_0)^{\Gamma(v)}Z(\hat{\bG})$.
  The product
  $$\beta(\gamma_0;\gamma,\delta)=\prod_v \beta_v(\gamma_0;\gamma,\delta)|_{\tilde{\fkK}(I_0/\Q)}
  \in X^*(\tilde{\fkK}(I_0/\Q))$$
  is trivial on $Z(\hat{\bG})$ by construction, hence defines a character
  $\alpha(\gamma_0;\gamma,\delta)$
  of $\fkK(I_0/\Q)=(\cap_v Z(\hat{I}_0)^{\Gamma(v)}Z(\hat{\bG}))/Z(\hat{\bG})$.

\section{Endoscopic transfer of test functions}\label{s:transfer}

  In this section we define the test functions needed to stabilize the formula given in
  Theorem \ref{MainTheorem1} below, following Kottwitz.

  First, we recollect some general facts about endoscopy. Let $F$ be a local
  or global field of characteristic 0.
  Let $\bG$ be a connected reductive group over $F$. As in the last section, we assume for convenience that
  \begin{itemize}
    \item  $\bG^{\der}$ is simply connected.
  \end{itemize}
  Let $(\bH,s,\eta)\in \cE(\bG)$ be an endoscopic triple consisting
  of a quasi-split connected reductive group $\bH$ over $F$, an element $s\in Z(\hat{\bH})$ and
  an embedding of complex Lie groups $\eta:\hat{\bH}\ra \hat{\bG}$ satisfying the conditions
  of \cite[7.4]{Kot84a}. Such a triple is elliptic if $(Z(\hat{\bH})^{\Gamma})^0\subset Z(\hat{\bG})$.
  The notion of isomorphism for endoscopic triples is defined in \cite[\S7]{Kot84a}.
  We use the Weil form of the $L$-group, and fix an $L$-morphism ${}^L \bH \ra {}^L \bG$
  extending $\eta$, which exists since $\bG^{\der}$ is simply connected (\cite[Prop 1]{Lan79}). The latter $L$-morphism
  is also denoted $\eta$ by abuse of notation.
  Write $\cE(\bG)$ (resp. $\cE_{\el}(\bG)$)
  for the set of isomorphism classes of endoscopic (resp. elliptic endoscopic)
  triples for $\bG$.

\subsection{Places $v\neq p,\infty$: Untwisted endoscopy}\label{sub:fund-lemma}

  Now assume that $F$ is a local field.
  Langlands and Shelstad (\cite{LS87})
  define a transfer factor $$\Delta:\bH(F)_{\semis,\GHreg}\times \bG(F)_{\semis} \ra \C,$$
  which is canonical up to a nonzero constant.
  Some basic properties are that $\Delta(\gamma_{\bH},\gamma)$ depends only on the stable
  conjugacy class of $\gamma_H$ and the $\bG(F)$-conjugacy class of $\gamma$ and that
  $\Delta(\gamma_{\bH},\gamma)=0$ unless $\gamma_{\bH}$ is $(\bG,\bH)$-regular and associated to $\gamma$.
  The fundamental lemma and the transfer conjecture, which are now theorems due to
  Ng\^{o}, Waldspurger and others (see \cite{Wal97}, \cite{Ngo10} and references therein), assert

\begin{thm}\label{t:Fund-Lemma}
  For each $f\in C^\infty_c(\bG(F))$, there exists a $f^{\bH}\in  C^\infty_c(\bH(F))$ such that
  for every $\gamma_{\bH}\in \bH(F)_{\semis,\GHreg}$,
  $$SO^{\bH(F)}_{\gamma_{\bH}}(f^{\bH})= \sum_{\gamma\in\bG(F)_{\semis}/\sim }
   \Delta(\gamma_{\bH},\gamma) e(\bG_\gamma) O^{\bG(F)}_\gamma(f)$$
where $\gamma$ runs over the set of representatives for all conjugacy classes in $\bG(F)_{\semis}$ and $\bG_\gamma$ denotes the centralizer of $\gamma$ in $\bG$. Now assume that $\bH$, $\bG$ and $\eta$ are unramified and that $f\in \cH^{\ur}(\bG(F))$.
Then $\Delta$ can be normalized such that $f^{\bH}=\eta^*(f)$ satisfies the above formula, where $\eta^*: \cH^{\ur}(\bG(F))\ra \cH^{\ur}(\bH(F))$ is
the natural induced morphism of unramified Hecke algebras. In particular, if $f$ is the idempotent associated to a hyperspecial maximal compact subgroup,
then one may take $f^{\bH}$ as the idempotent of a hyperspecial maximal compact subgroup as well.
\end{thm}

  We will say that $f^{\bH}$ as above is an ($\eta$-)transfer of $f$.

\begin{rem}\label{r:LS-normalization}
  Our transfer factor follows the convention of \cite{LS87} and \cite{KS99}, which are
  opposite to the one of \cite{KottwitzStabilization}. (See \cite[p.178]{KottwitzStabilization}
  and \cite[p.201]{Mor10}.) This introduces the sign change $s\mapsto s^{-1}$
  when citing results from the latter.
\end{rem}

\subsection{Place $p$: Twisted endoscopy}\label{sub:tw-endos-at-p}

  Let us recall the setup in which functions $\phi_{\tau,h}$ were defined in \cite{ScholzeDefSpaces}. In particular, we have a quasisplit
  reductive group $\bG$ over $\Q_p$ and a conjugacy class $\overline{\mu}$ of cocharacters
  $\mu: \mathbb{G}_m\ra \bG_{\bar{\Q}_p}$ with field of definition $E\subset \bar{\Q}_p$.\footnote{Here, we switch to Kottwitz'
  normalization of $\mu$: Let $\mu_{RZ}$ denote the $\mu$ considered in \cite{ScholzeDefSpaces}; then the product
  $\mu\mu_{RZ}$ is to be the central morphism $\G_m\ra \bG$ that sends $t\in \G_m$ to multiplication by $t$ on $V$.}
  We get $\mu_1\in X^*(Z(\hat{\bG}))$.
  Let $f$ be the inertial degree of the extension $E/\Q_p$, and let $r=jf$ for some $j\geq 1$. In \cite{ScholzeDefSpaces}, certain functions
  $\phi_{\tau,h}\in C_c^\infty(\bG(\Q_{p^r}))$ are defined, depending on $\tau$ in the Weil group of $E$, acting as the $j$-th power
  of geometric Frobenius on the residue field, and $h\in C_c^\infty(\bG(\Z_p))$ for a suitable integral model of $\bG$.
  One important property is that
  \begin{itemize}
  \item $\phi_{\tau,h}(\delta) = 0$ unless $\kappa_{\bG}(\delta) = - \mu_1$ in $X^\ast(Z(\hat{\bG})^{\Gamma(p)})$.
  \end{itemize}
  Here $\kappa_{\bG}: B(\bG)\ra X^*(Z(\hat{\bG})^{\Gamma(p)})$ is the Kottwitz morphism as above.

  We want to define the twisted endoscopic transfers of $\phi_{\tau,h}$. Let $(\bH,s,\eta)\in \cE(\bG_{\Q_p})$ be any endoscopic triple.
  First, we are going to give the construction of
  the endoscopic transfer $f^{\bH}_{\tau,h}$ of the function $\phi_{\tau,h}$ when $s\in Z(\hat{\bH})^{\Gamma(p)}$.
  It will be explained at the end of this subsection what to do in the general case where
  $s\in Z(\hat{\bH})^{\Gamma(p)}Z(\hat{\bG})$.
  Let $R_r=\Res_{\Q_{p^r}/\Q_p} \bG_{\Q_{p^r}}$, equipped with the obvious automorphism
  $\theta$ such that $R_r^{\theta}=\bG$.
  We have an identification $\hat{R}_r=\hat{\bG}\times \cdots\times \hat{\bG}$ ($r$ copies)
  on which any lift of the arithmetic Frobenius $\sigma$ acts by
  $(g_1,...,g_r)\mapsto (\sigma(g_2),...,\sigma(g_r),\sigma(g_1))$. In particular,
  there is a natural $L$-morphism $\xi: {}^L \bG\ra {}^L R_r$ sending $g\rtimes w\in {}^L \bG$ to
  $(g,\ldots,g)\rtimes w\in {}^L R_r$. Choose elements $t_i\in Z(\hat{\bH})^{\Gamma(p)}$ for $i=1,...,r$ such that
  $t_1t_2\cdots t_r=s$ and write $t=(t_1,...,t_r)$.
  Consider the $L$-morphism
  \begin{equation}\label{e:L-morphism-at-p}
    \tilde{\eta}:{}^L \bH \ra {}^L R_r
  \end{equation}
  defined by $\tilde{\eta}(h)= \xi(\eta(h))$ for $h\in \hat{\bH}$ and $\tilde{\eta}(\tilde{\sigma})=t\xi(\eta(\tilde{\sigma}))$
  for any lift $\tilde{\sigma}\in {}^L \bH$ of $\sigma$.
  Then $(\bH,t,\tilde{\eta})$ is a twisted endoscopic
  group for $(R_r,\theta)$.

  Ng\^{o} and Waldspurger (see \cite{Wal08}) showed
  that the twisted endoscopic transfer exists. The implication in our situation is that
  there exists a function $f^{\bH}_{\tau,h}\in C^\infty_c(\bH(\Q_p))$ such that
  for all $(\bG,\bH)$-regular semisimple $\gamma_{\bH}\in \bH(\Q_p)$,
 \begin{equation}
\label{e:SO-p0}
SO^{\bH(\Q_p)}_{\gamma_{\bH}}(f^{\bH}_{\tau,h})= \sum_{\delta\in \bG(\Q_{p^r})/\sim_\sigma}
\Delta_p(\gamma_{\bH},\delta)
  e(\bG_{\delta\sigma}) TO_{\delta\sigma}(\phi_{\tau,h}).\end{equation}
 where $\Delta_p$ denotes the twisted transfer factor, which is nonzero exactly when
 the twisted conjugacy class of $\delta$
 transfers to the stable conjugacy class of $\gamma_{\bH}$, and $\bG_{\delta\sigma}$ is the
 $\sigma$-centralizer of $\delta$ in $\bG$.

 As in our situation $\bG$ is quasi-split over $\Q_p$, more can be said about the transfer factors.
 Kottwitz showed (\cite[Appendix]{Mor10}) that
 if canonical transfer factors (\cite[A.1.5]{Mor10}) are used then
\begin{equation}
  \label{e:twisted-trans-factor}
  \Delta_p(\gamma_{\bH},\delta)= \lg \alpha(\gamma_0;\delta),s\rg^{-1}
  \Delta_p(\gamma_{\bH},\gamma_0),
\end{equation}
provided that $\gamma_0\in \bG(\Q_p)_{\semis}$ and $\gamma_{\bH}$ have
matching stable conjugacy classes. Such a $\gamma_0$ exists since $\bG$ is quasi-split over $\Q_p$.
Henceforth we will always choose $\Delta_p(\gamma_{\bH},\delta)$ and
  $\Delta_p(\gamma_{\bH},\gamma_0)$ to be canonical transfer factors.

  For later use, we recall the definition of $ \alpha(\gamma_0;\delta)$ here, cf. \cite[p.167]{KottwitzStabilization}.
  Let us write $I_0$ for the centralizer of $\gamma_0$ in $\bG$ over $\Q_p$.
  We have $I_0\subset \bH$ since $\gamma_0$ is $(\bG,\bH)$-regular.
  Recall that the set of $\sigma$-conjugacy classes in $I_0(L)$ is denoted $B(I_0)$.
  Using Steinberg's theorem, choose any $d\in \bG(L)$ such that $\delta \delta^{\sigma}\cdots \delta^{\sigma^{r-1}}
  = d \gamma_0 d^{-1}$. Then it turns out that $d^{-1}\delta d^{\sigma}$ belongs to $I_0(L)$
  and yields a well-defined element of $B(I_0)$.
  The image of the latter element under the Kottwitz map
  $\kappa_{I_0}:B(I_0)\ra X^*(Z(\hat{I_0})^{\Gamma(p)})$ is called $\alpha(\gamma_0;\delta)$.

  Putting this together, we get
  \begin{equation}\label{e:SO-p}
SO^{\bH(\Q_p)}_{\gamma_{\bH}}(f^{\bH}_{\tau,h})= \sum_{\delta\in \bG(\Q_{p^r})/\sim_\sigma}
 \lg \alpha(\gamma_0;\delta),s\rg^{-1} \Delta_p(\gamma_{\bH},\gamma_0)
  e(\bG_{\delta\sigma}) TO_{\delta\sigma}(\phi_{\tau,h}).\end{equation}

  So far we treated the case where $s\in Z(\hat{\bH})^{\Gamma(p)}$.
  In general, we may write $s=s_0 z$ for $s_0\in Z(\hat{\bH})^{\Gamma(p)}$ and $z\in Z(\hat{\bG})$.
  At this point, we remark that the twisted orbital integrals $TO_{\delta\sigma}(\phi_{\tau,h})$ vanish
  unless $\kappa_{\bG}(\delta) = - \mu_1$ in $X^\ast(Z(\hat{\bG})^{\Gamma(p)})$.
  This allows us to define $\beta(\gamma_0;\delta)
  \in X^*(Z(\hat{I_0})^{\Gamma(p)}Z(\hat{\bG}))$ as the extension of $\alpha(\gamma_0;\delta)$
  which is $-\mu_1$ on $Z(\hat{\bG})$.
  Then $\lg \beta(\gamma_0;\delta),s\rg^{-1}
  = \mu_1(z)\lg \beta(\gamma_0;\delta),s_0\rg^{-1}$.
  We define $f^{\bH}_{\tau,h}$ relative to $(\bH,s,\eta)$ as
  $\mu_1(z)$ times the function $f^{\bH}_{\tau,h}$ relative to $(\bH,s_0,\eta)$
  as constructed above, so that \eqref{e:SO-p} still holds if one restricts the sum to $\delta$
  satisfying $\kappa_{\bG}(\delta) = -\mu_1$ and replaces $\alpha(\gamma_0;\delta)$ by $\beta(\gamma_0;\delta)$.

\subsection{Place $\infty$: Pseudocoefficients}\label{sub:construction-infty}

  In this subsection, we assume that $F=\R$. Moreover, we fix an algebraic representation $\xi$ of $\bG$ over $\C$.
  Assume that $(\bH,s,\eta)$ is an elliptic endoscopic triple of $\bG$. We recall a construction of Kottwitz
  based on Shelstad's theory of real endoscopy.

  For a discrete series representation $\pi_\bH$ of $\bH(\R)$, write
  $\phi_{\pi_\bH}$ for its pseudo-coefficient. For any discrete
  $L$-parameter $\varphi_\bH:W_\R\ra {}^L \bH$, write
  $\Pi^2(\varphi_\bH)$ for its associated $L$-packet. Set
  $\phi_{\varphi_\bH}=|\Pi^2(\varphi_\bH)|^{-1} \sum_{\pi_{\bH}\in \Pi^2(\varphi_\bH)}
  \phi_{\pi_\bH}$.
  Write $\varphi_\xi:W_\R\ra {}^L \bG $ for the discrete $L$-parameter
  such that $\Pi^2(\varphi_\xi)$ consists of
  discrete series with the same central character and infinitesimal character as $\xi^\vee$.
  Define $\Phi_{\bH}(\varphi_\xi)$ to be the set of discrete $\varphi_\bH$
  such that $\eta\varphi_\bH$ and $\varphi_\xi$ are equivalent. Set
  $$f^{\bH}_{\xi}=
  \lg \mu_\bh, s \rg^{-1} \sum_{\varphi_\bH\in\Phi_{\bH}(\varphi_\xi)} (-1)^{q(\bG)}\det(\omega_*(\varphi_{\bH})) \phi_{\varphi_\bH}$$
  where $\det(\omega_*(\varphi_{\bH}))$ is defined on pp.184-185 and $\lg \mu_\bh,s\rg^{-1}$
  is defined on p.185 of \cite{KottwitzStabilization}. Kottwitz shows that for every $\gamma_{\bH}\in \bH(\R)_{\semis,\GHreg}$,
  \begin{equation}
\label{e:SO-infty}
  SO^{\bH(\R)}_{\gamma_{\bH}}(f^{\bH}_{\xi})
  = \lg \beta_\infty(\gamma_0), s \rg^{-1}\Delta_\infty(\gamma_{\bH},\gamma_0)
  e(I) \tr \xi(\gamma_0) \vol(A_{\bG}(\mathbb{R})^\circ\bs I(\R))^{-1}
\end{equation}
if $\gamma_{\bH}$ is elliptic and $SO^{\bH(\R)}_{\gamma_{\bH}}(f^{\bH}_\infty)=0$ otherwise, cf. \cite[(7.4), page 186]{KottwitzStabilization}.
Here $\Delta_\infty(\gamma_{\bH},\gamma_0)$  is as on page 184 of \cite{KottwitzStabilization}, and
$I$ is the inner form of the centralizer $I_0$ of $\gamma_0$ in $\bG$ that is anisotropic modulo center.
Moreover Proposition 3.3.4 (cf. Remark 6.2.2) of \cite{Mor10} implies that
  \begin{equation}
\label{e:SO-infty-vanish}
  SO^{\bH(\R)}_{\gamma_{\bH}}(f^{\bH}_{\xi}) = 0
\end{equation}
if $\gamma_{\bH}\in \bH(\R)_{\semis}$ is not $(\bG,\bH)$-regular.

\section{Stabilization - proof of Theorem \ref{MainTheorem2}}\label{s:proof-Thm1.2}

  Let us first briefly recall the main theorem from \cite{ScholzeDefSpaces}.
  \\
  Let $B$ be a simple $\mathbb{Q}$-algebra with center $F$ and maximal $\mathbb{Z}_{(p)}$-order $\mathcal{O}_B$ that is stable under a positive involution $\ast$ on $B$. Let $V$ be a finitely generated left $B$-module with a nondegenerate $\ast$-hermitian form $(\, ,\, )$. We assume that there is an $\mathcal{O}_B$-stable selfdual $\mathbb{Z}_{(p)}$-lattice $\Lambda\subset V$, which we fix. Moreover, we let $F_0 = F^{\ast = 1}$, which is a totally real field. We assume that at all places above $p$, the extension $F/F_0$ is unramified and the $F$-algebra $B$ is split.

We let $C=\mathrm{End}_B(V)$, and $\mathcal{O}_C = \mathrm{End}_{\mathcal{O}_B}(\Lambda)$; both carry an involution $\ast$ induced from $(\, ,\, )$. We recall, cf. \cite{KottwitzPoints}, p. 375, that over $\bar{\mathbb{Q}}$, the algebra $C$ together with the involution $\ast$ is of one of the following types:
\begin{enumerate}
\item[(A)] $M_n\times M_n^{\mathrm{opp}}$ with $(x,y)^\ast = (y,x)$,
\item[(C)] $M_{2n}$ with $x^\ast$ being the adjoint of $x$ with respect to a nondegenerate alternating form in $2n$ variables,
\item[(D)] $M_{2n}$ with $x^\ast$ being the adjoint of $x$ with respect to a nondegenerate symmetric form in $2n$ variables.
\end{enumerate}
We assume that case $A$ or $C$ occurs. We get the reductive group $\mathbf{G}/\mathbb{Q}$ of $B$-linear similitudes of $V$; in fact, we can extend it to an algebraic group over $\mathbb{Z}_{(p)}$ as the group representing the functor
\[
\mathbf{G}(R) = \{g\in (\mathcal{O}_C\otimes_{\mathbb{Z}_p} R)^\times \mid gg^\ast\in R^\times\}\ .
\]
Finally, we fix a homomorphism $\mathbf{h}_0: \mathbb{C}\rightarrow C\otimes \mathbb{R}$ such that $\mathbf{h}_0(\overline{z})=\mathbf{h}_0(z)^\ast$ for all $z\in \mathbb{C}$, and such that the symmetric real-valued bilinear form $(v,\mathbf{h}_0(i)w)$ on $V\otimes \mathbb{R}$ is positive definite.
\\

We write $\mathbf{h}$ for the map $\mathbb{S}\rightarrow \mathbf{G}\otimes \mathbb{R}$ from Deligne's torus $\mathbb{S}$ (i.e., the algebraic torus over $\mathbb{R}$ with $\mathbb{S}(\mathbb{R}) = \mathbb{C}^\times$) that is given on $\mathbb{R}$-valued points by $\mathbf{h}(z) = \mathbf{h}_0(z)$, $z\in\mathbb{C}^\times$. Then one gets a tower $\mathrm{Sh}_K$, $K\subset \mathbf{G}(\mathbb{A}_f)$ running through compact open subgroups of the adelic points of $\mathbf{G}$, of Shimura varieties associated to the pair $(\mathbf{G},\mathbf{h}^{-1})$.

Of course, we get $\mu$ and the reflex field $E$ as before, and $\mathrm{Sh}_K$ has a canonical model defined over $E$. We fix a prime $\mathfrak{p}$ of $E$ above $p$.

Finally, fix an algebraic representation $\xi$ of $\bG$, defined over a number field $L$, and let $\lambda$ be a place of $L$ lying over the rational prime $\ell\neq p$. This gives $\ell$-adic local systems $\mathcal{F}_{\xi,K}$ on $\mathrm{Sh}_K$ to which the action of the Hecke operators extend. Now we define
\[
H^\ast_{\xi} = \varinjlim_K H^\ast(\mathrm{Sh}_K\otimes \bar{\mathbb{Q}},\mathcal{F}_{K,\xi})
\]
as a virtual representation of $\Gal(\bar{\mathbb{Q}}/E)\times \bG(\A_f)$.

We make the assumption that the integral models defined in \cite{ScholzeDefSpaces} are proper; in particular, the Shimura varieties themselves need to be proper, i.e. $\bG$ needs to be anisotropic modulo center.

  We choose Haar measures on various groups exactly as in \cite[\S5.1]{Mor10}, cf. also \cite[\S3]{KottwitzStabilization}.
  We merely remark that adelic groups are given the Tamagawa measure and that
  the local measures at nonarchimedean places assign rational numbers to open compact subgroups.
  Moreover, in our situation we have defined a model of our group over $\mathbb{Z}_p$, and require that $\mathbf{G}(\mathbb{Z}_p)$ and $\mathbf{G}(\mathbb{Z}_{p^r})$ get volume $1$.

  Now we have the following theorem from \cite{ScholzeDefSpaces}.

  \begin{thm}\label{MainTheorem1} Let $\tau\in W_{E_\fkp}$ project to the $j$-th power of geometric Frobenius, for some $j\geq 1$. Let $h\in C_c^\infty(\bG(\Z_p))$ and $f^p\in C_c^\infty(\bG(\A_f^p))$. Then
  \[
  \tr(\tau\times hf^p|H^\ast_\xi) = \sum_{\substack{(\gamma_0;\gamma,\delta)\in \KT_j\\ \alpha(\gamma_0;\gamma,\delta) = 1}} c(\gamma_0;\gamma,\delta) O_\gamma(f^p) TO_{\delta\sigma}(\phi_{\tau,h}) \tr \xi(\gamma_0)\ ,
  \]
  where $c(\gamma_0;\gamma,\delta)$ is a volume constant defined as in \cite[p. 172]{KottwitzStabilization}.
  \end{thm}

  We remark that a priori these expressions are numbers in an algebraic closure of $L_\lambda$; we choose an isomorphism $\bar{L}_\lambda\cong \C$ and consider everything as $\C$-valued from now on.

  Just as (4.2) was obtained from (3.1) in \cite{KottwitzStabilization},
  we derive from Theorem \ref{MainTheorem1} that $\mathrm{tr}(\tau\times h f^p|H^\ast_\xi) $
  is equal to
\begin{equation}\label{e:pre-stab}
\begin{aligned}\tau(\bG) \sum_{\gamma_0}\sum_{\kappa} \sum_{(\gamma,\delta)}
\lg \alpha(\gamma_0;\gamma,\delta),\kappa\rg^{-1} e(&\gamma,\delta) O_\gamma(f^p)TO_{\delta\sigma}(\phi_{\tau,h})\\
&\tr \xi(\gamma_0) \vol(A_{\bG}(\R)^\circ\bs I(\infty)(\R))^{-1},
\end{aligned}\end{equation}
where the inner sum runs over $(\gamma,\delta)$ such that $(\gamma_0;\gamma,\delta)\in \KT_j$. Again, $I(\infty)/\R$ is
the compact modulo center inner form of the centralizer $I_0$ of $\gamma_0$ in $\bG$. For every $\kappa$ in the sum,
fix a lift $\tilde{\kappa}$ of $\kappa$ to $\tilde{\fkK}(I_0/\Q)$.
Note $\lg \alpha(\gamma_0;\gamma,\delta),\kappa\rg =\lg \beta(\gamma_0;\gamma,\delta),\tilde{\kappa}\rg$.

Lemma 9.7 of \cite{Kot86} and the global hypothesis
for transfer factors (\cite[\S6.4]{LS87})
may be summarized as follows.

\begin{lem}\label{l:transfer-conjugacy}
  There exists a canonical map
  $$
  \mathcal{T}: \coprod_{(\bH,s,\eta)\in \cE_{\el}(\bG)} \bH(\Q)_{\GHreg,\semis}/\sim_{\st}$$
  $$\longrightarrow \{(\gamma_0,\kappa):
  \gamma_0\in \bG(\Q)_{\semis}/\sim_{\st},~\kappa\in \fkK(I_0/\Q)\} \cup \{ \emptyset\}$$
  such that
\begin{enumerate}
  \item $(\bH,s,\eta,\gamma_{\bH})\mapsto (\gamma_0,\kappa)$ if
  $\gamma_0$ and $\gamma_{\bH}$ have matching stable conjugacy classes,
  $s$ lands in $\tilde{\fkK}(I_0/\Q)\subset Z(\hat{I}_0)$
   via $ Z(\hat{I}_{\gamma_{\bH}})\simeq Z(\hat{I}_0)$ and the image of $s$
    maps to $\kappa\in \fkK(I_0/\Q)$ via projection.
  \item $(\bH,s,\eta,\gamma_{\bH})\mapsto \emptyset$ if
  $\gamma_{\bH}$ does not transfer to $\bG(\Q)$.
\end{enumerate}
  For any $(\gamma_0,\kappa)$, there is a unique elliptic endoscopic triple $(\bH,s,\eta)\in \cE_{\el}(\bG)$
  containing a preimage of $(\gamma_0,\kappa)$; in that case, we have
  $$|\mathcal{T}^{-1}(\gamma_0,\kappa)|=|\mathrm{Out}((\bH,s,\eta))|.$$
  Moreover, for each
  $(\bH,s,\eta)\in \cE_{\el}(\bG)$, the transfer factors $\Delta_v$ may be normalized
  globally such that the following holds: for all
  $ \gamma_0\in \bG(\Q)_{\semis}$ and $\gamma_{\bH}\in \bH(\Q)_{\GHreg,\semis}$,
  $$\prod_v\Delta_v(\gamma_{\bH},\gamma_0)=
  \left\{
  \begin{array}{cl}
 1, &
  \mbox{if}~\gamma_{\bH}\mapsto\gamma_0,
   \\ 0, &
    \mbox{otherwise}\end{array} \right.$$
\end{lem}

  Thanks to the fundamental lemma (Theorem \ref{t:Fund-Lemma}),
  we can find
   an $\eta$-transfer $f^{\bH,p}$ of $f^p$ such that
  for every $\gamma_{\bH}\in \bH(\A^p_f)_{\semis,\GHreg}$,
 \begin{equation}
\label{e:SO-outside-p}
SO^{\bH(\A^p_f)}_{\gamma_{\bH}}(f^{\bH,p})= \sum_{\gamma\in \bG(\A^p_f)_{\semis}/\sim}
  \Delta^p(\gamma_{\bH},\gamma) e^p(\gamma) O_\gamma(f^p).\end{equation}
  Kottwitz's construction of $\beta_v$ at $v\neq p,\infty$ implies that
  (see the second last equality of \cite[p.169]{KottwitzStabilization},
  keeping Remark \ref{r:LS-normalization} in mind)
 \begin{equation}
\label{e:trans-outside-p}
  \Delta^p(\gamma_{\bH},\gamma)=
  \Delta^p(\gamma_{\bH},\gamma_0)\frac{\Delta^p(\gamma_{\bH},\gamma)}{\Delta^p(\gamma_{\bH},\gamma_0)}
  = \Delta^p(\gamma_{\bH},\gamma_0)\lg \beta^p(\gamma_0;\gamma,\delta),s\rg^{-1}.
\end{equation}
We have defined $f^{\bH}_{\tau,h}$ and $f^{\bH}_\xi$ in the last section.
 Put $f^{\bH}=f^{\bH,p}f^{\bH}_{\tau,h}f^{\bH}_\xi$.
 Lemma \ref{l:transfer-conjugacy}, formulas
  \eqref{e:SO-p}, \eqref{e:twisted-trans-factor},
   \eqref{e:SO-infty-vanish}, \eqref{e:SO-outside-p},
  \eqref{e:trans-outside-p} and \eqref{e:SO-infty}
  imply (cf. argument on \cite[p.188-189]{KottwitzStabilization}) that
   the expression \eqref{e:pre-stab} is equal to
$$ \sum_{(\bH,s,\eta)\in \cE_{\el}(\bG)}\tau(\bG)|\mathrm{Out}((\bH,s,\eta))|^{-1}
 \sum_{\substack{\gamma_{\bH}\in\bH(\Q)_{\semis}/\sim_{\st}\\
 \mathrm{elliptic~in~}\bH(\R)}} SO_{\gamma_H}(f^{\bH})$$
$$=\sum_{(\bH,s,\eta)} \iota(\bG,\bH)\tau(\bH) \sum_{\substack{\gamma_{\bH}\in\bH(\Q)_{\semis}/\sim_{\st}\\
 \mathrm{elliptic~in~}\bH(\Q)}} SO_{\gamma_H}(f^{\bH})
= \sum_{(\bH,s,\eta)} \iota(\bG,\bH) ST^{\bH}_{\el}(f^{\bH}).$$
Here $\iota(\bG,\bH)$ is by definition
$\tau(\bG)\tau(\bH)^{-1}|\mathrm{Out}((\bH,s,\eta))|^{-1}$.
The proof of Theorem \ref{MainTheorem2} is complete.

\begin{rem}\label{r:dependence-on-triple}
  The stable orbital integrals of the function $f^{\bH}$
  depend only on the isomorphism class of $(\bH,s,\eta)$. For instance
  suppose that $(\bH,s,\eta)$ is replaced with an isomorphic triple $(\bH,sz,\eta)$ with $z\in Z(\hat{\bG})$.
  Then $f^{\bH,p}$ does not change whereas
  $f^{\bH}_{\tau,h}$ and $f^{\bH}_\xi$ are multiplied by $\mu_1(z)$ and $\mu_1^{-1}(z)$, respectively.
\end{rem}

\section{Pseudostabilization}\label{s:pseudostab}

In this section, we employ the easier process of pseudostabilization as in \cite{KottwitzLambdaAdic}
that works for certain compact unitary group Shimura varieties with trivial endoscopy. It has the advantage that
one can avoid mention of endoscopy almost completely.\footnote{In particular, all of our results in the
nonendoscopic case are independent of recent work on the fundamental lemma. Only the base change fundamental lemma by Kottwitz, Clozel and Labesse is needed.}

We recall the setup, cf. \cite{KottwitzLambdaAdic} and Section 8.1.1.3 in \cite{Far04}. We start with a CM field $F$ with totally real subfield $F_0$.
Let $D/F$ be a central division algebra of dimension $n^2$ with an involution $\ast$ that restricts to complex conjugation on $F$.
We want to construct PEL data with $C=D$, compatible with $\ast$. First, we set $B=D^{\mathrm{op}}$, acting on $V=B$ from the left.
In order to define the pairing on $V$, we choose some $b\in B^\times$ with $b^\ast = b$ such that the pairing
$$(x,y) = \mathrm{tr}_B(xby^\ast)$$
gives a positive pairing on $B\otimes_\Q \R$; this is possible because the set of all $b\in B^{\ast=1}\otimes \mathbb{R}$ with this
property is an open cone in $B^{\ast=1}$ by Lemma 2.8 in \cite{KottwitzPoints}. We endow $B$ with the involution $x\mapsto x^\sharp = bx^\ast b^{-1}$;
the previous remarks show that this is a positive involution, as required.
Now take some $\beta\in F$ with $\beta^\sharp = -\beta$, and endow $V$ with the inner product
$$\lg x,y\rg = \mathrm{tr}_B(xb^{-1}\beta y^\sharp).$$
One checks that starting with $(B,\sharp)$, this makes $V$ a left $B$-module with a nondegenerate $\sharp$-hermitian form,
and $C=\mathrm{End}_B(V)=D$, inducing the involution $\ast$ on $C=D$.
Finally, we take some $\ast$-homomorphism $\bh_0: \C\ra D\otimes \R$ such that $\lg x,\bh_0(i)y\rg$ is a positive pairing on $V$.

This defines global PEL data. Let $E$ be the associated reflex field, $p$ a prime satisfying our local assumptions; in particular,
we have additional integral data at $p$, and the group $\bG$, defined over $\Q$ by
\[
\bG(R) = \{ x\in D\otimes_\Q R\mid xx^\ast\in R^\times\}
\]
has a model over $\mathbb{Z}_{(p)}$. Also, let $\fkp$ be a place of $E$ above $p$, and let $\xi$ be an algebraic representation
of $\bG$ over a number field $L$; let $\lambda$ be a place of $L$ lying above $\ell\neq p$.

\begin{thm}\label{MainTheorem3} Let $\tau$, $h$ and $f^p$ as usual. Then
\[
N \tr(\tau\times hf^p | H^\ast_\xi) = \tr(f_{\tau,h}^{\bG} f^p | H^\ast_\xi)\ .
\]
Here $N=|\Pi^2(\xi)|\cdot |\pi_0(G(\R)/Z(G)(\R))|$, cf. \cite{KottwitzLambdaAdic}, p. 659, and $f_{\tau,h}^{\bG}$
is the endoscopic transfer of $\phi_{\tau,h}$ for the trivial endoscopic triple $(\bH,s,\eta)=(\bG,1,\mathrm{id})$.
\end{thm}

\begin{proof} This follows from following the arguments in \cite{KottwitzLambdaAdic}. The crucial point in comparison to the general case is that $\mathfrak{K}(I_0/\mathbb{Q})$ is trivial for all Kottwitz triples $(\gamma_0;\gamma,\delta)$. This makes the sum over $\kappa$ in \eqref{e:pre-stab} superfluous, and one can factor everything into stable orbital integrals directly. The outcome is, cf. \cite[(5.2)]{KottwitzLambdaAdic}, that the left-hand side equals
\[
N\tau(\bG) \sum_{\gamma_0} SO_{\gamma_0}(f_{\tau,h}^{\bG} f^p f_\infty)\ ,
\]
where $\gamma_0$ runs through stable conjugacy classes in $\bG(\Q)$. Applying the trace formula (and Lemma 4.1 of \cite{KottwitzLambdaAdic}), this rewrites as
\[
N \sum_\pi m(\pi) \tr \pi(f_{\tau,h}^{\bG} f^p f_\infty)\ .
\]
Now one applies Lemma 4.2 of \cite{KottwitzLambdaAdic} to rewrite this as
\[
\sum_{\pi_f} \tr \pi_f(f_{\tau,h}^{\bG} f^p) \sum_{\pi_\infty} m(\pi_f\otimes \pi_\infty) \mathrm{ep}(\pi_\infty\otimes \xi)\ ,
\]
with the notation $\mathrm{ep}$ denoting the Euler-Poincar\'{e} characteristics of the Lie algebra cohomology as in \cite{KottwitzLambdaAdic}, p. 660.
Finally, one uses Matsushima's formula, cf. \cite{KottwitzLambdaAdic}, p. 655, to rewrite this as
\[
\tr(f_{\tau,h}^{\bG} f^p | H^\ast_\xi)\ ,
\]
as desired.
\end{proof}

In general one would expect that the following statement is true at places $p$, such that all places of $F_0$ above $p$ are split in $F$: If $\pi_f$, $\pi_f^\prime$ are two irreducible admissible representations of $\bG(\A_f)$ such that the $\pi_f$-isotypic and $\pi^\prime_f$-isotypic components of $H^\ast_\xi$ are nonzero and the representations are isomorphic away from $p$, i.e. $\pi_f^p\cong \pi_f^{\prime p}$, then also $\pi_p\cong \pi_p^\prime$.

Under some extra conditions, this can deduced from results of Harris and Labesse, \cite{HarrisLabesse}, which show the existence of a stable base change lift of automorphic representations $\pi$ of $\bG$ to automorphic representations of $\GL_n$ over $F$, cf. \cite{Far04}, Section A.3.

In particular, assume that $F$ is the composite of $F_0$ and an imaginary quadratic field $\cK$. Note that if $q$ is a rational prime that splits in $\cK$ as $w_0w_0^c$, there is a decomposition
\[
\bG(\Q_p) = \prod_{v|w_0} \GL_n(F_v)\times \Q_p^\times\ ,
\]
which induces a decomposition
\[
\pi_q = \bigotimes_{v|w_0} \pi_v \otimes \chi_{w_0}
\]
for every irreducible admissible representation $\pi_q$ of $\bG(\Q_p)$.

\begin{thm}\label{t:cond-bc} Assume that $F=F_0 \cK$ is the composite of the totally real field $F_0$ and an imaginary quadratic field $\cK$ in which $p$ splits. Assume that the division algebra $D$ is split at all places of $F_0$ which do not split in $F$, and at all nonarchimedean places either split or a division algebra. Let $\pi_f$ be an irreducible admissible representation of $\bG(\A_f)$ that occurs in $H^\ast_\xi$, and such that there is some rational prime $q=w_0w_0^c$ split in $\cK$ such that in the decomposition
\[
\pi_q = \bigotimes_{v|w_0} \pi_v \otimes \chi_{w_0}\ ,
\]
one component $\pi_v$ is supercuspidal.

Then there exists a cuspidal automorphic representation $\Pi$ of $\GL_n(\A_F)$ and a Hecke character $\chi$ of $\A_{\cK}^\times$ such that $\Pi_\infty$ is regular algebraic, $\Pi^\vee\cong \Pi\circ c$, where $c:F\ra F$ is complex conjugation, and for all places $p=u_0u_0^c$ split in $\cK$, the decomposition of $\pi_p$ from above is given by
\[
\pi_p = \bigotimes_{v|u_0} \Pi_v\otimes \chi_{u_0}\ .
\]

A similar description holds true at unramified inert places.
\end{thm}

\begin{proof} This follows from Theorem A.3.1 in \cite{Far04}.
\end{proof}

\begin{cor}\label{cor:RigidAtp} In the situation of the theorem, let $p\neq q$ be split in $\cK$ and let $\pi_f^\prime$ be a second irreducible admissible representation of $\bG(\A_f)$ which occurs in $H^\ast_\xi$ and for which $\pi_f^{\prime p}\cong \pi_f^p$. Then $\pi_p\cong \pi_p^\prime$.
\end{cor}

\begin{proof} Use strong multiplicity $1$ for $\GL_n$ and the quadratic base change for $\bG$ in the nonendoscopic case. cf. Corollary VI.2.3 of \cite{HT01}.
\end{proof}

Moreover, the results of Harris-Taylor, \cite[Theorem C]{HarrisTaylor}, prove the existence of $\ell$-adic Galois representations associated to $\Pi$ that satisfy a local-global compatibility at all finite places not dividing $\ell$. We remark that for our purposes, the slightly weaker statement proven in \cite{ScholzeLLC} would suffice.

\begin{cor}\label{cor:ExGalRepr} In the situation of the theorem, there is an $\ell$-adic Galois representation
\[
R_\ell(\Pi): \Gal(\bar{F}/F)\ra \GL_n(\bar{\Q}_\ell)
\]
such that for all primes $v$ of $F$ not dividing $\ell$, the restriction $R_\ell(\Pi)|_{W_{F_v}}$ is given by $\sigma_\ell(\Pi_v^\vee)\otimes |\cdot|^{(1-n)/2}$, where $\sigma_\ell$ denotes the local Langlands correspondence for $\GL_n(F_v)$.$\hfill \Box$
\end{cor}

\section{The stable Bernstein center}\label{s:stBernstein}

The most convenient way to state our conjecture is to use the conjectural theory of the stable Bernstein center, cf. \cite{Vog93}. The importance of the stable Bernstein center in relation to test functions appearing in generalizations of the Langlands-Kottwitz method has also been emphasized recently by Haines.

Let us recall the main definitions. Let $G$ be a connected reductive group over a $p$-adic field $F$. In our convention
varieties are not required to be of finite type or connected.

\begin{defprop}(\cite{Ber84}) A supercuspidal pair for $G$ is a pair $(M,\sigma)$, where $M\subset G$ is a Levi subgroup of $G$, and $\sigma$ is a supercuspidal representation of $M(F)$. An irreducible smooth representation $\pi$ of $G(F)$ is said to have supercuspidal support $(M,\sigma)$, if it appears as a subquotient of the normalized parabolic induction of $\sigma$ to $G(F)$; for any $\pi$, there is a unique such $(M,\sigma)$ up to conjugation by $G(F)$.

The variety $\Omega(G)$ of all supercuspidal pairs $(M,\sigma)$ up to conjugation by $G(F)$ has a natural structure as the $\mathbb{C}$-valued points of an infinite disjoint union of algebraic varieties over $\mathbb{C}$, which are quotients of tori by finite group actions. Let $\mathcal{Z}(G)$ be the ring of regular functions on $\Omega(G)$, the Bernstein center of $G$. Then there is a natural $\mathcal{Z}(G)$-algebra structure on $C_c^\infty(G(F))$, given by a convolution map
\[
\mathcal{Z}(G)\times C_c^\infty(G(F))\rightarrow C_c^\infty(G(F))\ :\ (z,f)\mapsto z\ast f\ ,
\]
such that for all $f\in C_c^\infty(G(F))$, all $z\in \mathcal{Z}(G)$ and any irreducible $\pi$ with supercuspidal support $(M,\sigma)$, we have
\[
\tr(z\ast f|\pi) = z(M,\sigma) \tr(f|\pi)\ .
\]
\end{defprop}

\begin{defprop} A semisimple $L$-parameter\footnote{Vogan named it an infinitesimal character in \cite{Vog93}, which unfortunately conflicts with the common usage of the term as a character of $\mathcal{Z}(G)$, which Vogan refers to as a classical infinitesimal character.} for $G$ is a homomorphism $\lambda: W_F\ra {}^L G$ from the Weil group $W_F$ of $F$ into the $L$-group of $G$ that commutes with the projection maps to the Galois group of $F$, and such that $\lambda(W_F)$ consists of semisimple elements.

The variety $\Omega^{\st}(G)$ of all semisimple $L$-parameters for $G$ up to $\hat{G}$-conjugation has a natural structure as an infinite disjoint union of algebraic varieties over $\mathbb{C}$, which are quotients of tori by finite group actions. Let $\mathcal{Z}^\st(G)$ be the ring of regular functions on $\Omega^\st(G)$, the stable Bernstein center of $G$.
\end{defprop}

Now assume for motivation that the local Langlands conjecture is known. In that case, any irreducible $\pi$ has an associated $L$-parameter $\varphi_\pi: W_F\times \SL_2(\C)\ra {}^L G$, and the restriction $\varphi_\pi|_{W_F}$ defines a semisimple $L$-parameter $\lambda_\pi$. Here, we embed $W_F$ into $W_F\times \SL_2(\C)$ by using the identity map on the first factor, and the map
\[
w\mapsto \left(\begin{array}{cc}|w|_F^{1/2} & 0 \\ 0 & |w|_F^{-1/2}\end{array}\right)
\]
from $W_F$ to $\SL_2(\C)$, using the norm character $|\cdot|_F:W_F\ra \R^\times_{>0}$. Conjecturally, cf. Conjecture 7.18 in \cite{Vog93}, this induces a finite map of algebraic varieties
\[
\Omega(G)\ra \Omega^\st(G)\ :\ (M,\sigma)\mapsto \lambda(M,\sigma)
\]
such that whenever $\pi$ has supercuspidal support $(M,\sigma)$, then $\lambda_\pi = \lambda(M,\sigma)$.

We will use the following form of the local Langlands conjecture. cf. Remark \ref{r:Kazhdan} below.

\begin{conjecture}\label{conj:st-Bernstein} There is a natural finite map of algebraic varieties
\[
\Omega(G)\ra \Omega^\st(G)\ .
\]
The induced map $\mathcal{Z}^\st(G)\ra \mathcal{Z}(G)$ has the following property: If $f\in C_c^\infty(G(F))$ has vanishing stable orbital integrals then for any $z^\st \in \mathcal{Z}^\st(G)$ with image $z\in \mathcal{Z}(G)$, the convolution $z\ast f\in C_c^\infty(G(F))$ has vanishing stable orbital integrals.
\end{conjecture}

The conjecture is true for (products of Weil restrictions of) $\GL_n$, by \cite{HT01}, \cite{Hen00}. In the case of $GL_n$ one of us essentially constructed the map $\mathcal{Z}^\st(G)\ra \mathcal{Z}(G)$ directly from Lubin-Tate towers in \cite{ScholzeGLn}.

Moreover, let us recall some statements from \cite{Vog93} about stable characters. Here we only explain the case where $G$ is quasisplit and refer the reader to \cite{Vog93} or \cite{Art06} for general $G$. The easiest case is the case of tempered $L$-parameters; so let
\[
\varphi: W_F\times \SL_2(\C)\ra {}^L G
\]
be a tempered $L$-parameter. Let $S_\varphi$ be the group of $g\in \hat{G}$ centralizing the image of $\varphi$ up to a 1-coboundary of $W_F$ valued in $Z(\hat{G})$. Define $\fkS_{\varphi}=\pi_0(S_{\varphi}/Z(\hat{G}))$.
  It is expected that there exists a bijection $\pi\mapsto r_\pi$
  from the $L$-packet $\Pi(\varphi)$ to the set of irreducible finite dimensional representations of $\fkS_{\varphi}$. Moreover, one expects that the distribution
\[
 S\Theta_{\varphi}=\sum_{\pi\in \Pi(\varphi)} \dim r_\pi \tr \pi
\]
  is stable. The bijection $\pi\mapsto r_\pi$ is supposed to be canonical up to twist by a character, so that the stable character is independent of all choices.

  In particular, in this case, a suitable linear combination of the characters of the representations in the $L$-packet induces a stable character. In general, this will not be the case, and one has to include some representations from other $L$-packets, as in the example of $A$-packets. However, all representations appearing will have the same {\it semisimple} $L$-parameter. Indeed, Vogan conjectures that there are many stable characters of the form
  \begin{equation}\label{e:gen-st-char}
  S\Theta = \sum_{\pi} a(\pi) \tr \pi\ ,
  \end{equation}
  where $a(\pi)$ is nonzero for only finitely many $\pi$, and for all such $\pi$, the semisimple $L$-parameter $\lambda=\lambda_\pi$ is the same. Let $z^\st\in \mathcal{Z}^\st(G)$, with associated $z\in \mathcal{Z}(G)$. Then for all $f\in C_c^\infty(G(F))$, we have
  \begin{equation}\label{e:factor-st-Bernstein}
  S\Theta(z\ast f) = z^\st(\lambda) S\Theta(f)\ .
  \end{equation}
  Indeed, this is an immediate consequence of Conjecture \ref{conj:st-Bernstein} and the formula for the action of the Bernstein center on the Hecke algebra.

\begin{rem}\label{r:Kazhdan}
  The work of Kazhdan and Varshavsky \cite{KV06} suggests a harmonic analytic
  definition of the stable Bernstein center without reference to semisimple $L$-parameters.\footnote{In fact our definition is not the same
  as the one in \cite{KV06}. Though the two definitions are supposed to be equivalent, and this is the case for $G=GL_n$, it is difficult to check in general. Our definition has the advantage that $\cY^{\st}(G)$ is easily seen to be a subalgebra of $\cZ(G)$.}
   Recall that $\cZ(G)$ can be viewed as the algebra of $G(F)$-invariant distributions $\alpha$ on $G(F)$ which preserve $C^\infty_c(G(F))$ under the convolution $*$. Define $\cY^{\st}(G)$ to be the subspace of $\cZ(G)$ consisting of $\delta$ such that $\delta* f$ has vanishing stable orbital integrals
  for every $f\in C^\infty_c(G(F))$ whose stable orbital integrals vanish. It is clear that $\cY^{\st}(G)$ is a subalgebra of $\cZ(G)$.
  In this optic Conjecture \ref{conj:st-Bernstein} would hold true if the following are verified: $\cZ(G)$ is a finitely generated $\cY^{\st}(G)$-module and there is a natural surjection $\cZ^{\st}(G)\ra \cY^\st(G)$, which should be induced by an infinitesimal version of the local Langlands correspondence. If $G$ is quasisplit, then the map $\cZ^\st(G)\ra \cY^\st(G)$ should be an isomorphism. This anticipated isomorphism justifies our terminology for $\cZ^{\st}(G)$.
\end{rem}

\section{Character identities satisfied by $f^{\bH}_{\tau,h}$: Conjectures}\label{s:char-id}

  The contents of this section are purely local in nature, and we go back to the
  local setting of \cite{ScholzeDefSpaces} and Subsection \ref{sub:tw-endos-at-p}.
  Let $(\bH,s,\eta)$ be an endoscopic triple for $\bG$.

  We will state two version of our conjecture. One relies on the local Langlands correspondence for $\bH$, at least for tempered $L$-parameters. The other relies on Conjecture \ref{conj:st-Bernstein} about the stable Bernstein center, and is stronger.

  Recall that we have a conjugacy class $\ol{\mu}$ of cocharacters $\mu:\G_m\ra \bG_{\ol{\Q}_p}$ as defined in Subsection
  \ref{sub:tw-endos-at-p}. Fix a maximal torus $\bT$ of $\bG$ over $\ol{\Q}_p$ with a set of positive roots.
  After conjugation we may assume that $\mu$ factors through $\bT$ and
  is a dominant coweight. Write $\rho$ for the half sum of all positive roots.
  Up to isomorphism there exists a unique complex representation $r_{-\mu}$ of ${}^L \bG$
  characterized by \cite[Lem 2.1.2]{Kot84b}. In particular, $r_{-\mu}|_{\hat{\bG}}$
  is irreducible, with highest weight given by (the dominant representative of) $-\mu$.
  Let $\mathrm{Frob}$ denote a geometric Frobenius element in $W_E$.

\begin{conjecture}\label{c:char-id}
  Let $\tau\in \mathrm{Frob}^j I_{E}$ for $j\ge 1$ and $h\in C^\infty_c(\bG(\Z_p))$.
  Denote by $h^{\bH}$ an $\eta$-transfer of $h$ from $\bG$ to $\bH$.
  For every tempered $L$-parameter $\varphi_{\bH}$ of $\bH$, with associated
  semisimple $L$-parameter $\lambda_{\bH} = \varphi_{\bH}|_{W_{\Q_p}}$, we have
  \begin{equation}\label{e:char-id-for-f-tau-h}
    S\Theta_{\varphi_{\bH}}(f^{\bH}_{\tau,h})=
  \tr (s^{-1}\tau|(r_{-\mu}\circ \eta\lambda_{\bH}|_{W_E})|\cdot |_E^{-\lg \rho,\mu\rg}) S\Theta_{\varphi_{\bH}}(h^{\bH}).
  \end{equation}

 \end{conjecture}

 Here is a simple consistency check.
 Let $z\in Z(\hat{\bG})$. Suppose that $(\bH,s,\eta)$ is replaced with $(\bH,sz,\eta)$ and that
  $f^{\bH}_{\tau,h}$ is constructed from the latter. Then both sides of \eqref{e:char-id-for-f-tau-h}
  are multiplied by $\mu_1(z)$ (cf. Remark \ref{r:dependence-on-triple}).

  It is expected that stable tempered characters are dense in the space of stable distributions,
  cf. \cite{Sha90}, Conjecture 9.2. Namely for any stable distribution $\mathcal{D}$ on $\bH(\Q_p)$, the value
  of $\mathcal{D}(f^{\bH}_{\tau,h})$ should be determined by \eqref{e:char-id-for-f-tau-h} for tempered
  $\varphi_{\bH}$. In particular, there should be an analogue of \eqref{e:char-id-for-f-tau-h} for nontempered $\varphi_{\bH}$.
  We feel that the best way to formulate this conjecture is to use the stable Bernstein center.

  Let $z_\tau^{\bH,\st}\in \mathcal{Z}^\st(\bH)$ be the function in the stable Bernstein center of $\bH$ that takes
  each semisimple $L$-parameter $\lambda: W_{\Q_p}\ra {}^L \bH$ to
  \[
   \tr (s^{-1}\tau|(r_{-\mu}\circ \eta\lambda|_{W_E})|\cdot |_E^{-\lg \rho,\mu\rg})\ .
   \]
   General requirements on the local Langlands correspondence ensure that this is a regular function on $\Omega^\st(\bH)$.
   Let $z_\tau^\bH\in \mathcal{Z}(\bH)$ be its image.

  \begin{conjecture}\label{c:char-id-st-Bernstein}
  Let $\tau$ and $h$ be as above. Then for any $\eta$-endoscopic transfer $h^{\bH}$ of $h$, the function $z_\tau^\bH\ast h^{\bH}$ is
  a twisted endoscopic $\tilde{\eta}$-transfer of $\phi_{\tau,h}$.
  \end{conjecture}

  This conjecture says that the twisted endoscopic transfers of $\phi_{\tau,h}$ factor as a product of two factors, one depending only on $\tau$, the other depending only on $h$. This property by itself would be of great interest.

  As a consequence of this conjecture, assume that $S\Theta$ is a stable distribution as in \eqref{e:gen-st-char}. Let $f_{\tau,h}^\bH$ denote any
  transfer of $\phi_{\tau,h}$. Then
  \begin{equation}\label{e:gen-char-id-for-f-tau-h}
  S\Theta(f_{\tau,h}^\bH) =  \tr (s^{-1}\tau|(r_{-\mu}\circ \eta\lambda|_{W_E})|\cdot |_E^{-\lg \rho,\mu\rg}) S\Theta(h^{\bH})\ .
  \end{equation}
  This generalizes \eqref{e:char-id-for-f-tau-h}, and should apply in particular to the stable characters associated to $A$-packets.

  If $\bG(\Q_p)$ is a product of general linear groups over finite extensions of $\Q_p$, as happens in all cases of EL or quasi-EL type
  (cf. Remark \ref{r:quasi-EL} below),
  then $\bH(\Q_p)$ is of the same type, and the relevant local Langlands conjecture is known by \cite{HT01} and \cite{Hen00}.
  In this case the statements of Conjecture \ref{c:char-id} and Conjecture \ref{c:char-id-st-Bernstein} do not rely on any unverified conjecture
  and simplify as stability issues do not arise. It is an easy exercise to verify that the two conjectures are equivalent.
  The conjecture was shown to be true by one of us, \cite{ScholzeGLn}, in the Lubin-Tate case.
  In \S\ref{s:prove-conj} we will settle Conjecture \ref{c:char-id} in all EL cases.

A version of this conjecture also makes sense in the unramified case without any assumption
  thanks to the unramified Langlands correspondence.
  Hence assume that we are in the case of unramified EL or PEL data. Given our assumptions,
  this amounts to the only additional requirement that $F$ is unramified.
  In particular, $\bG$ is unramified over $\Q_p$. Suppose moreover that $\bH$, $\eta:{}^L \bH \ra {}^L \bG$,
   and $h$ are unramified; in particular (up to scalar) $h=\triv_{\bG(\Z_p)}$.
  If $\varphi_{\bH}$ is unramified, then we write $\pi_{\varphi_{\bH}}$ for the unramified
  representation corresponding to $\varphi_{\bH}$. We define a character $S\Theta_{\varphi_{\bH}}$ on $\cH^{\ur}(\bH(\Q_p))$,
   by setting $S\Theta_{\varphi_{\bH}}$ equal to $\tr \pi_{\varphi_{\bH}}$ if $\varphi_{\bH}$ is unramified and $0$ otherwise.

\begin{lem}\label{l:conj-unramified}
 Assume that $F$ is unramified, and that $p\neq 2$ in the PEL case. One can choose $f_{\tau,h}^\bH,h^{\bH}\in \cH^{\ur}(\bH(\Q_p))$,
 and for all $L$-parameters $\varphi_{\bH}$ with semisimple $L$-parameter $\lambda_{\bH}$, we have
 \begin{equation}\label{e:unr-case}
 S\Theta_{\varphi_{\bH}}(f_{\tau,h}^{\bH}) =  \tr (s^{-1}\tau|(r_{-\mu}\circ \eta\lambda_{\bH}|_{W_E})|\cdot |_E^{-\lg \rho,\mu\rg}) S\Theta_{\varphi_{\bH}}(h^{\bH})\ .
 \end{equation}
\end{lem}

\begin{rem}
  The lemma shows that \eqref{e:char-id-for-f-tau-h} naturally
  extends the identity in the unramified case treated in \cite{KottwitzStabilization}.
\end{rem}

\begin{proof}
  From Proposition 4.7 of \cite{ScholzeDefSpaces}, we know that $\phi_{\tau,h}$ has the same twisted orbital integrals as
  Kottwitz's $\phi_r$ on \cite[p.173]{KottwitzStabilization} when $h=\triv_{\bG(\Z_p)}$. We may thus work with $\phi_r$ instead.
  We note that $\phi_r$ lies in the unramified Hecke algebra $\cH^{\ur}(\bG(\Q_{p^r}))$.

  The $L$-morphism \eqref{e:L-morphism-at-p} gives rise to
  the transfer $\tilde{\eta}^*:\cH^{\ur}(\bG(\Q_{p^r}))\ra \cH^{\ur}(\bH(\Q_p))$
  such that $\phi\mapsto \tilde{\eta}^*(\phi)$ is an $\tilde{\eta}$-transfer
  with respect to the canonical transfer factor of \S\ref{sub:tw-endos-at-p} (cf. \cite[A.1.5]{Mor10}).

  In particular, we can choose $f_{\tau,h}^{\bH},h^{\bH}\in \cH^{\ur}(\bH(\Q_p))$, as desired.
  Moreover, both sides of \eqref{e:unr-case} vanish unless $\varphi_{\bH}$ is unramified.

  Now assume that $\varphi_{\bH}$ is unramified. Then the argument for (9.7) of \cite{KottwitzStabilization}, which is based
  on the Satake transform (thus unconditional), shows that
  $$S\Theta_{\varphi_{\bH}}(f^{\bH}_{\tau,h})=
  p^{r\lg \rho,\mu\rg} \tr (s^{-1}\tau|(r_{-\mu}\circ \eta\lambda_{\bH}|_{W_E})),$$
  which is equivalent to \eqref{e:unr-case}.
  Note that $\Delta_p(\psi_{\bH},\pi_{\eta\varphi_{\bH}})=1$ in (9.7) of \cite{KottwitzStabilization} since the above transfer factor
  is compatible with the algebra map of unramified Hecke algebras.
\end{proof}

  As the final topic of this section, and as further supporting evidence for the conjecture, we verify that it leads
  to the expected description of the cohomology. We switch back to global notation and
  assume that $\bG$ is anisotropic over $\Q$ modulo center so that our Shimura varieties are compact.\footnote{We also assume that the flat closure of the generic fibre in the integral models defined in \cite{ScholzeDefSpaces} is proper.}

  The following are accepted on faith:
\begin{enumerate}
  \item the $A$-packet classification for discrete automorphic
  representations of $\bG$ and its
  elliptic endoscopic groups over $\Q$, encompassing the
  multiplicity formula \cite[(8.2)]{KottwitzStabilization}
  and a product formula for spectral transfer factors \cite[(9.6)]{KottwitzStabilization},\footnote{In particular, we assume that there is the global Langlands group. This should be avoidable in all PEL cases by using alternative global parameters as in \cite{ArthurEndoscopy}.}
  \item $ST^{\bH}_{\el}(f^{\bH})=ST^{\bH}_{\disc}(f^{\bH})$
  for all $(\bH,s,\eta)\in \cE_{\el}(\bG)$. (cf. \cite[(10.1)]{KottwitzStabilization})
\end{enumerate}
  Our argument copies the one in
   \cite[\S9-10]{KottwitzStabilization} with changes occurring only at $p$.
   Therefore we will only indicate the changes without repeating the details. All references
   are made to Kottwitz' article.
\begin{itemize}
  \item Use $f^{\bH,p}$, $f^{\bH}_{\tau,h}$ and $f^{\bH}_\xi$
  in place of $h^p$, $h_p$ and $h_\infty$.
  \item Drop the assumption that $\psi_p$ is unramified.
  Now (9.3) is equal to (9.7) if the latter is replaced by
  $$\sum_{\pi_p\in \Pi(\eta\psi_{\bH,p})} \Delta_p(\psi_{\bH},\pi_p)
  \tr (s^{-1}\tau|(r_{-\mu}\circ \eta\varphi_{\bH}|_{W_{E_{\fkp}}})|\cdot|_{\fkp}^{-\dim \Sh/2}) \tr \pi_p(h).$$
  This key equality is justified by Conjecture \ref{c:char-id}, or more precisely its version for $A$-parameters.
  \item On page 198, drop conditions (a) and (c).
  \item (9.9) is valid if $\tr \pi_f(f)$ is changed to
  $\tr \pi_f^p(f^p)$ and the line below reads
  $$A(x,\pi_p)=(-1)^{q(\bG)} \tr (s^{-1}\tau|(r_{-\mu}\circ \varphi_{\pi_p}|_{W_{E_{\fkp}}})|\cdot|_{\fkp}^{-\dim \Sh/2})
  \tr \pi_p(h).$$
  \item (10.4) holds if $\tr \pi_f(f)$
  is replaced with $\tr \pi_f^p(f^p)$ and
    we set in the line below
   $$A(\pi_p,\nu)=(-1)^{q(\bG)} \nu(s_\psi)\tr (\tau|V_\nu|\cdot|_{\fkp}^{-\dim \Sh/2})
   \tr \pi_p(h).$$
  \item From (10.4) we obtain (cf. (10.5)) that
  $\tr (\tau\times hf^p|H^*_\xi)$ equals
  $$ \sum_{[\psi]}\sum_{\pi_f} \tr \pi_f(hf^p)\sum_{\nu}
  (-1)^{q(\bG)} m(\pi_f,\nu)  \nu(s_\psi) \tr (\tau|V(\psi,\nu)).$$
  \item The above equality establishes that
  $$H^*_\xi=\sum_{[\psi]}\sum_{\pi_f} \sum_{\nu}
 (-1)^{q(\bG)} m(\pi_f,\nu) \nu(s_\psi) (\pi_f\otimes V(\psi_p,\nu)).$$
  as virtual $\bG(\A^p_f)\times \bG(\Z_p)\times W_{E_{\fkp}}$-representations, as desired.
\end{itemize}

Of course, one would expect that this identity stays true as $\bG(\A_f)\times W_{E_\fkp}$-representations. If $\bG_{\Q_p}$ is a product of Weil restrictions of general linear groups, then $\pi_p$ should be determined by $\pi_f^p$, which easily implies that this identity extends to $\bG(\A_f)\times W_{E_\fkp}$. In the general case, the first obstruction is that if $\pi_f$ and $\pi_f^\prime$ are two representations appearing in $H^\ast_\xi$ such that $\pi_f^p\cong \pi_f^{\prime p}$, then it is not clear that $\pi_p$ and $\pi_p^\prime$ belong to the same $A$-packet. Ignoring this obstacle, the question of whether one can isolate representations within a local $A$-packet by Hecke operators in $\bG(\Z_p)$ appears:

\begin{question} Let $G$ be a connected reductive group over a $p$-adic field $F$, and let $K\subset G(F)$ be a special maximal compact subgroup. Let $\psi$ be an $A$-parameter for $G$, with associated $A$-packet $\Pi(\psi)=\{\pi_1,\ldots,\pi_k\}$. Are $\pi_1|_K,\ldots,\pi_k|_K$ linearly independent in the Grothendieck group of admissible representations of $K$?
\end{question}

\section{Character identities satisfied by $f^{\bH}_{\tau,h}$: Proofs}\label{s:prove-conj}

  Let us keep the local notation of the last section (except that global notation
  will appear in the proof of Theorem \ref{t:conj-GLn}).
  Recall that the local Langlands correspondence for (a product of) general linear groups
  over a $p$-adic field is known thanks to Harris-Taylor and Henniart (\cite{HT01}, \cite{Hen00}).
  In what follows $\varphi_{\pi}$ denotes the $L$-parameter attached
  to an irreducible representation $\pi$ on such a group.

\begin{thm}\label{t:conj-GLn}
  Conjecture \ref{c:char-id} is true in all EL and quasi-EL cases for the trivial endoscopic triple $(\bH,s,\eta)=(\bG,1,\mathrm{id})$.
  In fact, for every irreducible representation $\pi$ of $\bG(\Q_p)$ with semisimple $L$-parameter $\lambda_\pi$,
\begin{equation}\label{e:conj-GLn} \tr(f^{\bG}_{\tau,h}|\pi)= \tr (\tau|(r_{-\mu}\circ \lambda_{\pi}|_{W_E})
|\cdot |_E^{-\lg \rho,\mu\rg}) \tr (h|\pi).\end{equation}
\end{thm}

\begin{rem}\label{r:quasi-EL}
 We recall that in \cite{ScholzeDefSpaces}, quasi-EL data were defined as those PEL data where the field $F$ decomposes as $F_0\times F_0$ with the involution $\ast$ acting by $(x,y)^\ast = (y,x)$. We note that there is a bijection between EL data and quasi-EL data, and that quasi-EL data come about by localization of global PEL data of type A at a split place.
\end{rem}

\begin{proof}
The conjectures in an EL case and the corresponding quasi-EL case are equivalent by Proposition 4.10 of \cite{ScholzeDefSpaces}.
Moreover, the conjecture in the EL case reduces to the simple EL case by Proposition 4.9 of \cite{ScholzeDefSpaces}.
By Morita equivalence, we may moreover assume that $V=B$ as left $B$-module, where $B$ is a matrix algebra over $F$.

Summarizing, we may assume that we are in the quasi-EL case corresponding to a simple EL case,
where $B=B_0\times B_0$ is a matrix algebra over $F=F_0\times F_0$, following the notation for the quasi-EL case in \cite{ScholzeDefSpaces}.
The involution on $B$ is given by $(x,y)^\ast = (y,x)$, and $V=B$ as left $B$-module.

  For the proof we may fix $\tau$ and $h$. Also fix an algebraic representation $\xi$ of $\bG$ of regular highest weight.
  As our proof will be global in nature, we will adopt global notation as needed.
  Let us globalize the local datum over $\Q_p$ to a global datum $(B,\cO_B,F,V,\Lambda,C,\cO_C,\bh)$
  and a $*$-hermitian form on $V$ for a compact unitary Shimura variety with trivial endoscopy as in
  Section \ref{s:pseudostab}, satisfying the assumptions of Theorem \ref{t:cond-bc} as done in
  \cite[Prop 8.1.3]{Far04}. In particular $F$ is a CM field containing an imaginary quadratic field $\cK$, $p$ splits in $\cK$, $F$ decomposes
  as the product of its totally real subfield $F_0$ and $\cK$, and $p$ is inert in $F_0$, with its local field $F_{0,p}$ being
  isomorphic to what was previously called $F_0$. After basechange from $\Q$ to $\Q_p$, we get the previous local PEL data
  and there exists a prime $\fkp$ of the reflex field $E$ of the datum
  such that the completion $E_{\fkp}$ is isomorphic to the local reflex field $E$ occuring in \eqref{e:conj-GLn},
  compatible with the cocharacter $\mu$. Let $\bG$ denote the reductive group over $\Z_{(p)}$ for the global datum,
  and $\pi_p$ denote $\pi$ of \eqref{e:conj-GLn}.

  The first step is to verify \eqref{e:conj-GLn} assuming
  \begin{itemize}
    \item $\pi_p$ appears as the $p$-component of an irreducible automorphic representation
  $\pi$ of $\bG(\A)$ such that $\pi_\infty\in \Pi^2(\varphi_\xi)$ and $\pi_v$ is supercuspidal
  for some prime $q\neq p$ split in $\cK$ and some place $v|q$ of $F$.
  \end{itemize}

  Let
  $$H^\ast_\xi(\pi_f) = \Hom_{\bG(\A_f)}(\pi_f,H^\ast_\xi).$$
  Using Corollary \ref{cor:RigidAtp}, it follows from Theorem \ref{MainTheorem3} that
  \begin{equation}\label{e:global-trace}
  N \tr(\tau | H^\ast_\xi(\pi_f)) \tr(h|\pi_p) = \dim H^\ast_\xi(\pi_f) \tr(f_{\tau,h}^{\bG} | \pi_p).
  \end{equation}

  If $\pi_p$ is unramified, one may take $h=\triv_{\bG(\Z_p)}$ and use Lemma \ref{l:conj-unramified} to determine the restriction
  $H^\ast_\xi(\pi_f)|_{W_{E_\fkp}}$, cf. Theorem 1 in \cite{KottwitzLambdaAdic}. This gives
 \begin{equation}\label{e:local-global-compat}
 H^*_\xi(\pi_f)|_{W_{E_\fkp}}=a(\pi_f) (r_{-\mu}\circ \lambda_{\pi_p}|_{W_{E_{\fkp}}})|\cdot|_{\fkp}^{-\dim\Sh/2}
 \end{equation}
 for some integer $a(\pi_f)\in \Z$. This equation implies in particular that $N= \dim r_{-\mu}$ and $\dim H^\ast_\xi(\pi_f) = a(\pi_f) N$.
 We also recall that $\dim \Sh = \lg 2\rho,\mu\rg$. This argument works for any unramified split place $p$ (not only the one fixed in advance).
 Using Corollary \ref{cor:ExGalRepr}, one can build a
 representation of $\Gal(\bar{\Q}/E \cK)$ that agrees with $H^*_\xi(\pi_f)|_{\Gal(\bar{\Q}/E \cK)}$
 at all unramified split places; we refer to the proof of \cite[Thm A.7.2]{Far04} for the precise construction.
 From the Chebotarev density theorem, it follows that they agree everywhere, which shows that
 \eqref{e:local-global-compat} continues to hold at ramified split places, in particular at our given $\fkp$.

 Now one may combine \eqref{e:global-trace} with \eqref{e:local-global-compat} to deduce \eqref{e:conj-GLn}.

  It remains to get rid of the hypothesis on $\pi_p$. Consider the two linear forms
  $\cF_1$ (resp. $\cF_2$) on the Grothendieck group
  of admissible representations of $\bG(\Q_p)$ mapping each irreducible $\pi_p$
  to the value of the left (resp. right) hand side of \eqref{e:conj-GLn}.
  It is easy to see that both $\cF_1$ and $\cF_2$ are trace functions in the context
  of the trace Paley-Wiener theorem of \cite{BDK86}. (This is a tautology for $\cF_1$.)
  Since \cite[Thm 4.4]{Shin-Plan} shows that a dense subset of the unitary dual of $\bG(\Q_p)$
  satisfies the above hypothesis, we know that $\cF_1=\cF_2$ on that dense subset.
  Since $\cF_1$ and $\cF_2$ are regular functions on (an infinite disjoint union of) algebraic varieties,
  we deduce that $\cF_1(\pi_p)=\cF_2(\pi_p)$ for all $\pi_p$.
\end{proof}

At this point, let us give a corollary of the proof. We note that the identities of Conjecture \ref{c:char-id}
determine the orbital integrals of $f_{\tau,h}^\bG$ uniquely. On the other hand, the orbital integrals
of $f_{\tau,h}^\bG$ determine the twisted orbital integrals of $\phi_{\tau,h}$ uniquely. But, cf. the remarks
after Theorem 5.7 of \cite{ScholzeDefSpaces}, there was some freedom in the choice of $\phi_{\tau,h}$.
This gives the following corollary.

\begin{cor}\label{cor:AllPDivGroupsOccur} Let $\Sh$ be compact unitary group Shimura variety without endoscopy as in Section \ref{s:pseudostab},
satisfying the hypothesis of Theorem \ref{t:cond-bc}. Let $p$ be a place split in $\cK$, let $\fkp|p$ a place of $E$ and let $\mathcal{M}_{K^p}/\mathcal{O}_{E_\fkp}$ be the integral model of (a finite disjoint union of copies of) $\Sh_{\bG(\Z_p)K^p}$ over $\mathcal{O}_{E_\fkp}$ defined in \cite{ScholzeDefSpaces}.

Let $\mathcal{D}$ be the corresponding local data of quasi-EL type, associated to data $\mathcal{D}_0$ of EL type. Then any
$p$-divisible group $\underline{\overline{H}}$ with $\mathcal{D}$-structure over $\bar{\mathbb{F}}_p$ which admits a deformation to characteristic $0$ occurs as the $p$-divisible group with $\mathcal{D}$-structure associated to some point over $\bar{\mathbb{F}}_p$ on $\mathcal{M}_{K^p}$, for some $K^p$.
\end{cor}

\begin{proof} Note that a $p$-divisible group $\underline{\overline{H}}$ with $\mathcal{D}$-structure which admits a deformation to characteristic $0$, i.e. such that $X_{\overline{\underline{H}}}\neq \emptyset$, is parametrized by some $\delta\in \bG(W(\bar{\F}_p)[\frac 1p])$, cf. Proposition 3.10 of \cite{ScholzeDefSpaces}.

We can define two functions $\phi_{\tau,h}^{(0)}$ and $\phi_{\tau,h}^{(1)}$, both satisfying the conditions after Theorem 5.7
of \cite{ScholzeDefSpaces}, but such that $\phi_{\tau,h}^{(0)}$, resp. $\phi_{\tau,h}^{(1)}$, is $0$, resp. $1$, on all $\delta$ for which
nothing is imposed by these conditions. We want to prove that $\phi_{\tau,h}^{(0)} = \phi_{\tau,h}^{(1)}$.

Now by what we have shown, necessarily all twisted orbital integrals of $\phi_{\tau,h}^{(1)} - \phi_{\tau,h}^{(0)}$ vanish. But this is the characteristic function of some open bounded subset, and it follows that this subset has to be trivial.
\end{proof}

In the case that the local PEL data are unramified, this implies nonemptiness of Newton strata. Recall that the Newton stratification is a stratification of $\mathcal{M}_{K^p}\otimes \bar{\mathbb{F}}_p$ into locally closed subvarieties $\mathcal{M}_{K^p,b}$, $b\in B(\bG_{\Q_p},\mu)$, parametrized by a certain subset of $\sigma$-conjugacy classes $B(\bG_{\Q_p},\mu)\subset B(\bG_{\Q_p})$, cf. the work of Rapoport and Richartz, \cite{RapoportRichartz}. The question of nonemptiness of Newton strata is a deep question in the theory of the reduction of Shimura varieties. We refer to the survey article of Rapoport, \cite{RapoportGuide}, for an extensive discussion of related work. Recently, Viehmann and Wedhorn, \cite{ViehmannWedhorn}, have used local geometric methods to prove nonemptiness of Newton strata in many cases, including those that we handle here. Our method is entirely different, in that it uses global and automorphic arguments, similar to earlier work of Fargues, \cite{Far04}, showing that the basic locus is nonempty.

\begin{cor}\label{cor:NewtonNonempty} Assume that in the situation of Corollary \ref{cor:AllPDivGroupsOccur}, the field $F$ is unramified at $p$, so that $\bG$ is unramified at $p$ and $\bG(\Z_p)$ is a hyperspecial maximal compact subgroup. Then for all $b\in B(\bG_{\Q_p},\mu)$, the Newton stratum $\mathcal{M}_{K^p,b}$ is nonempty.
\end{cor}

\begin{proof} We only have to show that there is some $p$-divisible group with $\mathcal{D}$-structure that has the correct Newton polygon. In the notation of \cite{RapoportGuide}, this question is equivalent to the question $X(b,\mu)_K\neq \emptyset$ by Corollary 3.12 of \cite{RapoportGuide}. This is resolved by a theorem of Kottwitz and Rapoport in the case considered here, \cite{KottwitzRapoportExFCristals}, cf. Theorem 5.5 in \cite{RapoportGuide}.
\end{proof}

Another interesting application of Corollary \ref{cor:AllPDivGroupsOccur} is the following algebraization result that works in the ramified case, and seems to be hard to prove by other methods.

\begin{cor}\label{cor:ELAlg} Let $\mathcal{D}$ be data of EL type, and let $\underline{\overline{H}}$ be any $p$-divisible
group with $\mathcal{D}$-structure over $\bar{\mathbb{F}}_p$. Then the deformation spaces
$X_{\underline{\overline{H}},K}/X_{\underline{\overline{H}}}$ defined in \cite{ScholzeDefSpaces}
are algebraizable as in Theorem 2.9 of \cite{ScholzeDefSpaces}.
\end{cor}

\begin{proof} It suffices to consider the case of simple EL data, as in the general case everything decomposes
into a product. By Morita equivalence, we can assume that $V=B$ again. Changing to quasi-EL data, we may use
Corollary \ref{cor:AllPDivGroupsOccur} to embed the situation into a Shimura variety, which provides the desired algebraization.
\end{proof}

Now we continue with the proof of Conjecture \ref{c:char-id} in EL cases.

\begin{cor}\label{c:conj-GLn}
    Conjecture \ref{c:char-id} is true in all EL and quasi-EL cases
    for general endoscopic triples $(\bH,s,\eta)$ of $\bG$, i.e. for all irreducible representations $\pi_{\bH}$ of $\bH(\Q_p)$ with
    semisimple $L$-parameter $\lambda_{\pi_\bH}$, we have
    \[
    \tr(f^{\bH}_{\tau,h}|\pi_{\bH})=
  \tr (s^{-1}\tau|(r_{-\mu}\circ \eta\lambda_{\pi_\bH}|_{W_E})|\cdot |_E^{-\lg \rho,\mu\rg}) \tr(h^{\bH}|\pi_{\bH})\ .
  \]
\end{cor}

\begin{proof}
  Let $K$ be a finite extension of $\Q_p$. It suffices to consider the case of simple EL data with
  $B=K$, $V=K^n$, where $\bG\simeq \Res_{K/\Q_p} \GL_n$. In this case, any endoscopic triple
  $(\bH,s,\eta)$ of $\bG$ is of the form $\bH\simeq \Res_{K/\Q_p} (\prod_i \GL_{n_i})$
  for some $n_i\ge 1$ such that $\sum_i n_i=n$ and
 $$\eta:(\prod_i \GL_{n_i}(\C))^{\Hom(K,\ol{\Q}_p)}\rtimes W_{\Q_p}
   \ra \GL_n(\C)^{\Hom(K,\ol{\Q}_p)}\rtimes W_{\Q_p},$$
  is the canonical Levi embedding, i.e. we identify $\bH$ and $\hat{\bH}$ with
  Levi subgroups of $\bG$ and $\hat{\bG}$ via block diagonal embeddings.
  Moreover we may assume that $s\in Z(\hat{\bH})^{\Gamma(p)}$ as the general case is reduced to this
  case (cf. the last paragraph of \S\ref{sub:tw-endos-at-p}).
  Write $s=(s_i)_i\in Z(\hat{\bH})^{\Gamma(p)}=\prod_i \C^\times$
  (which embeds in the center of $(\prod_i \GL_{n_i}(\C))^{\Hom(K,\ol{\Q}_p)}$
  diagonally). Choose $t_0=(t_{0,i})_i\in Z(\hat{\bH})^{\Gamma(p)}$
  such that $t_0^r=s$. Let $c:W_{\Q_p}\ra Z(\hat{\bH})^{\Gamma(p)}$ be the
  cocycle $w\mapsto t_0^{-\ord(w)}$, where we normalize $\ord: W_E\ra \Z$ by requiring
  that it sends a geometric Frobenius element to $1$. Define $\eta':{}^L \bH \ra {}^L \bG$
  by $\eta'(g\rtimes w)=c(w)\eta(g\rtimes w)$.
  Recall that $R_r$ and $\tilde{\eta}$ were defined in \S\ref{sub:tw-endos-at-p}.
  Write $\xi:{}^L G \ra {}^L R_r$ for the base change $L$-morphism sending
  $g\rtimes w$ to $(g,g,...,g)\rtimes w$. Then it is straightforward to verify that
  \begin{equation}
    \tilde{\eta}=\xi\circ \eta'.
  \end{equation}
  The function $f^{\bH}_{\tau,h}$ is a twisted transfer of $\phi_{\tau,h}$ relative to $\tilde{\eta}$.
  We claim that it is also an $\eta'$-transfer of $f_{\tau,h}^\bG$ (denoted $(\eta')^*(f_{\tau,h}^\bG)$ below).

  Let us accept the claim for now and finish the proof.
  Note that an $\eta$-transfer of $f_{\tau,h}^\bG$, denoted $\eta^*(f_{\tau,h}^\bG)$,
   is given by the constant term along a parabolic
  subgroup with $\bH$ as a Levi component. It satisfies a character identity
   $$\tr (\eta^*(f_{\tau,h}^\bG)|\pi_{\bH})=\tr\left(f_{\tau,h}^\bG\left|\nind^{\bG(\Q_p)}_{\bH(\Q_p)}(\pi_{\bH})\right.\right)$$
  for every irreducible admissible representation $\pi_{\bH}$ of $\bH(\Q_p)$, where
  $\nind$ denotes the normalized parabolic induction (independent of the choice of parabolic).
  Let $\chi_c:H(\Q_p)\ra \C^\times$ be the character corresponding to $c\in H^1(W_{\Q_p},
  Z(\hat{\bH})^{\Gamma(p)})$.

  Since $\eta'=\eta\cdot c$, one can show that
  $(\eta')^* f^{\bG}_{\tau,h}$ may be given as $\chi^{-1}_c\cdot \eta^*(f_{\tau,h}^\bG)$.
  (The reason is that the Langlands-Shelstad transfer factor gets multiplied by $\chi^{-1}_c$
  when $\eta$ is multiplied by $c$.)
  As we can take $f^{\bH}_{\tau,h}=(\eta')^* f_{\tau,h}^\bG$ by the claim above,
  $$f^{\bH}_{\tau,h}(\gamma)=\eta^*(f_{\tau,h}^\bG)(\gamma)\chi^{-1}_c(\gamma), \quad\forall \gamma\in \bH(\Q_p).$$
  To put it in an invariant way,
    $$\tr(f^{\bH}_{\tau,h}|\pi_{\bH})=\tr(\eta^*(f_{\tau,h}^\bG)|\pi_{\bH}\otimes \chi_c)
  =\tr\left(f_{\tau,h}^\bG\left|\nind^{\bG(\Q_p)}_{\bH(\Q_p)}(\pi_{\bH}\otimes \chi_c)\right.\right).$$
  Note that all subquotients $\pi$ of $\nind^{\bG(\Q_p)}_{\bH(\Q_p)}(\pi_{\bH}\otimes \chi_c)$ give
  rise to the same semisimple $L$-parameter $\lambda_\pi = \eta\lambda_{\pi_{\bH}\otimes\chi_c}$. Hence, by
  \eqref{e:conj-GLn} the right hand side equals
   $$\tr(\tau|(r_{-\mu}\circ \eta\lambda_{\pi_{\bH}\otimes\chi_c}|_{W_{E}})) \tr\left(h\left|\nind^{\bG(\Q_p)}_{\bH(\Q_p)}(\pi_{\bH}\otimes \chi_c)\right.\right)$$
  $$=\tr(\tau|(r_{-\mu}\circ \eta\lambda_{\pi_{\bH}\otimes \chi_c}|_{W_{E}})) \tr(\eta^*(h)|\pi_{\bH}\otimes \chi_c).$$
  We recall that $c(\tau)=t_0^{-\ord(\tau)}=t_0^{-r}=s^{-1}$, from which
  it is not difficult to deduce that the twist by $\chi_c$ has the same effect as
  replacing $\tau$ by $s^{-1}\tau$:
  $$\tr(\tau|(r_{-\mu}\circ \eta\lambda_{\pi_{\bH}\otimes \chi_c}|_{W_E}))
  = \tr(s^{-1}\tau|(r_{-\mu}\circ \eta\lambda_{\pi_{\bH}}|_{W_E})).$$
  On the other hand, note that
  $$ \tr(\eta^*(h)|\pi_{\bH}\otimes \chi_c)= \tr(\eta^*(h)|\pi_{\bH}).$$
  Indeed, the constant term $\eta^*(h)$ is supported only on the conjugacy classes of $\bH(\Q_p)$ meeting
  $\bH(\Z_p)$, since $h$ is supported on $\bG(\Z_p)$.
  Since the character $\chi_c$ is trivial on such conjugacy classes (because $c$ is an unramified
  cocycle), the identity follows.

  Putting the above identities together, we conclude that
  $$\tr(f^{\bH}_{\tau,h}|\pi_{\bH})=\tr(s^{-1}\tau|(r_{-\mu}\circ \eta\lambda_{\pi_{\bH}}|_{W_E}))\tr(\eta^*(h)|\pi_{\bH}),$$
  which is the equality \eqref{e:char-id-for-f-tau-h} to be proved.

  It remains to justify the unproven claim that $f^{\bH}_{\tau,h}$ is an $\eta'$-transfer of $f_{\tau,h}^\bG$.
  For any $(\bG,\bH)$-regular semisimple $\gamma_{\bH}\in \bH(\Q_p)$, if $\gamma_{\bH}$ is not a transfer of
  any twisted conjugacy class of $\delta\in \bG(\Q_{p^r})$ then
  $SO_{\gamma_{\bH}}(f^{\bH}_{\tau,h})=SO_{\gamma_{\bH}}((\eta')^* f_{\tau,h}^\bG)=0$.
  From now on suppose that $\gamma_{\bH}$ is a transfer of some $\delta$. We note that in the cases that occur here,
  taking a stable (twisted) orbital integral amounts to multiplying the usual (twisted) orbital integral by
  a sign, given as the Kottwitz sign attached to the (twisted) centralizer of the element considered.
  By \eqref{e:SO-p}, \eqref{e:twisted-trans-factor},
$$
SO^{\bH(\Q_p)}_{\gamma_{\bH}}(f^{\bH}_{\tau,h})= \Delta^{\eta}_p(\gamma_{\bH},\gamma)
\lg \alpha(\gamma;\delta),s\rg^{-1}
  STO_{\delta\sigma}(\phi_{\tau,h}).$$
  The superscript $\eta$ indicates that the transfer factor is relative to $\eta$.
  On the other hand, Theorem \ref{t:Fund-Lemma} and its analogue for the base change relative to $\xi$ tell us that
$$
SO^{\bH(\Q_p)}_{\gamma_{\bH}}((\eta')^* f_{\tau,h}^\bG)= \Delta^{\eta'}_p(\gamma_{\bH},\gamma) SO^{\bG(\Q_p)}_{\gamma}(f_{\tau,h}^\bG)
= \Delta^{\eta'}_p(\gamma_{\bH},\gamma) STO_{\delta\sigma}(\phi_{\tau,h}).$$

  As we have noted above, $\Delta^{\eta'}_p(\gamma_{\bH},\gamma)=
  \chi^{-1}_c(\gamma_{\bH})\Delta^{\eta}_p(\gamma_{\bH},\gamma)$. Hence the proof boils down to showing that
  \begin{equation}\label{e:claim-conjGLn}\chi_c(\gamma_{\bH})=\lg \alpha(\gamma;\delta),s\rg.\end{equation}

  Define $R_r^{\bH}=\Res_{\Q_{p^r}/\Q_p} \bH$. We can find a $\delta_{\bH}\in R_r^{\bH}(\Q_p)=\bH(\Q_{p^r})$
  which is $\sigma$-conjugate to $\delta$ and transfers to $\gamma_{\bH}$ via base change for $\bH$.
  We write $\delta_{\bH}=(\delta_{\bH,i})_i$ and $\gamma_{\bH}=(\gamma_{\bH,i})_i$ with respect to $\bH\simeq \Res_{K/\Q_p}(\prod_i GL_{n_i})$.
  As usual $v_p$ denotes the additive valuation on a $p$-adic field such that $v_p(p)=1$.
  It can be seen from the definition of $\chi_c$ that
  $$\chi_c(\gamma_{\bH})=\prod_i t_{0,i}^{v_p(N_{K/\Q_p}\det \gamma_{\bH,i})}.$$
  As $\gamma_{\bH,i}$ is a norm of $\delta_{\bH,i}$, we get
  $$\chi_c(\gamma_{\bH})=\prod_i s_i^{v_p(N_{K/\Q_p}\det \delta_{\bH,i})}.$$

  In order to evaluate the right hand side of \eqref{e:claim-conjGLn},
  we begin by recalling the definition of $\alpha(\gamma;\delta)$.
  We are free to replace $\delta$ by a $\sigma$-conjugate, and hence we assume that $\delta=\delta_{\bH}$.
  Let $L$ be the fraction field of $W(\bar{\mathbb{F}}_p)$. Using Steinberg's theorem, one finds $d\in \bG(L)$
  such that $N\delta = d\gamma_{\bH} d^{-1}$. In fact, it is possible to take $d\in \bH(L)$, by choice of $\delta$, and
  as $\gamma_{\bH}$ is a base-change transfer of $\delta_{\bH}$.
  In general, $d^{-1}\delta d^\sigma\in I_0(L)$, where $I_0$ is the centralizer of $\gamma$ in $\bG$, and $\alpha(\gamma;\delta)$ is
  defined as the image of $d^{-1}\delta d^\sigma$ under Kottwitz' map $\kappa_{I_0}: B(I_0)\ra X^\ast(Z(\hat{I_0})^\Gamma)$.

  The Kottwitz maps
  $\kappa_{I_0}$ and $\kappa_{\bH}$ fit into a commutative diagram
  $$\xymatrix{ B(I_0) \ar[d]_-{\kappa_{I_0}} \ar[r] & B(\bH) \ar[d]_-{\kappa_{\bH}}
  \ar@{=}[r] & \prod_i B(\Res_{K/\Q_p} \GL_{n_i})\ar[d]_-{(\kappa_{\Res_{K/\Q_p} \GL_{n_i}})_i}   \\
  X^*(Z(\hat{I_0})^\Gamma) \ar[r]  & X^*(Z(\hat{\bH})^\Gamma)
  \ar@{=}[r] & \prod_i X^*(\C^\times)
  }$$
  The image of $\alpha(\gamma;\delta)$ in $X^\ast(Z(\hat{\bH})^\Gamma)$ is the image under $\kappa_{\bH}$
  of the $\sigma$-conjugacy class of $d^{-1} \delta_{\bH} d^\sigma$,
  i.e. the $\sigma$-conjugacy class of $\delta_{\bH}$, as $d\in \bH(L)$.

  Then $$ \lg \alpha(\gamma;\delta),s\rg = \lg \kappa_{\bH}(\delta_{\bH}),s\rg
  = \prod_i \lg \kappa_{\Res_{K/\Q_p} \GL_{n_i}}(\delta_{\bH,i}),s_i\rg.$$
  The right hand side is equal to $\prod_i s_i^{v_p(N_{K/\Q_p}\det \delta_{\bH,i})}$
  thanks to Lemma \ref{l:Kottwitz-map} below.
  We have finished the proof of \eqref{e:claim-conjGLn}.
\end{proof}

  The following well-known fact was used in the proof.
\begin{lem}\label{l:Kottwitz-map}
  Let $G=\Res_{K/\Q_p} \GL_n$ for a finite extension $K$ of $\Q_p$. Then the Kottwitz map
$$B(G)\ra X^*(Z(\hat{G})^{\Gamma(p)})=X^*(\C^\times)
\stackrel{\mathrm{canon.}}{\simeq} \Z$$
is induced by the map $g\mapsto v_p(N_{K/\Q_p}\det g)$ from $G(L)=\GL_n(K\otimes_{\Q_p} L)$ to $ \Z$.
\end{lem}

\begin{proof}
  The functoriality of the Kottwitz map (\cite[4.9]{Kot97})
  with respect to $$\det:\Res_{K/\Q_p}\GL_n\ra \Res_{K/\Q_p}\GL_1$$
  reduces the proof to the case of $n=1$, where the assertion is standard, cf. \cite[2.5]{Kot85}.
\end{proof}

\section{Cohomology of Shimura varieties: Nonendoscopic case}\label{s:lgc}

Consider a compact unitary group Shimura variety with trivial endoscopy as in Section \ref{s:pseudostab}.

\begin{thm}\label{t:nonendoscopic} Assume that all places of $F_0$ above $p$ are split in $F$. Then we have an identity
\[
H^\ast_\xi = \sum_{\pi_f} a(\pi_f) \pi_f\otimes (r_{-\mu}\circ \varphi_{\pi_p}|_{W_{E_\fkp}})|\cdot|^{-\dim \Sh/2}
\]
as virtual $\bG(\A_f^p)\times \bG(\Z_p)\times W_{E_\fkp}$-representations. Here $\pi_f$ runs through irreducible admissible representations of $\bG(\A_f)$, the integer $a(\pi_f)$ is as in \cite{KottwitzLambdaAdic}, p. 657, and $\varphi_{\pi_p}$ is the local Langlands parameter associated to $\pi_p$.
\end{thm}

\begin{proof} First, we note that the assumption implies that Conjecture \ref{c:char-id} is true in this case, by Theorem \ref{t:conj-GLn}. Matsushima's formula can be reformulated as the identity
\[
H^\ast_\xi = N \sum_{\pi_f} a(\pi_f) \pi_f
\]
as $\bG(\A_f)$-representation, cf. Lemma 4.2 of \cite{Kot92b}. Now Theorem \ref{MainTheorem3} and the identity from Conjecture \ref{c:char-id} show that
\[\begin{aligned}
\tr(\tau\times hf^p|H^\ast_\xi) &= N^{-1} \tr(f_{\tau,h}^\bG f^p | H^\ast_\xi) \\
&= \sum_{\pi_f} a(\pi_f) \tr(\tau|(r_{-\mu}\circ \varphi_{\pi_p}|_{W_{E_\fkp}})|\cdot |^{-\dim \Sh/2}) \tr(hf^p|\pi_f)\ .
\end{aligned}\]
This gives the desired identity.
\end{proof}

\begin{cor}\label{c:hasse-weil} In the situation of the theorem, let $K\subset \bG(\A_f)$ be any sufficiently small compact open subgroup. Then the semisimple local Hasse-Weil zeta function of $\Sh_K$ at the place $\fkp$ of $E$ is given by
\[
\zeta_\fkp^\semis (\Sh_K,s) = \prod_{\pi_f} L^\semis(s-\dim \Sh/2, \pi_p, r_{-\mu})^{a(\pi_f)\dim \pi_f^K}\ .
\]
\end{cor}

\begin{proof} One can assume that $K\subset \bG(\A_f^p)\times \bG(\Z_p)$. Now the corollary follows directly from the previous theorem and the definitions.
\end{proof}

\section{Cohomology of Shimura varieties: Endoscopic case}\label{s:compact}

  In this section we proceed to study compact unitary group Shimura varieties at split places when endoscopy is nontrivial.
  We will work in the setup of \cite[\S5.1]{Shin11} to be recalled now.
  Suppose that
\begin{itemize}
  \item $B=F$ is a CM field with complex conjugation as the involution $*$, and
  \item $F$ contains an imaginary quadratic field $\cK$.
\end{itemize}
  Fix a prime $p$ split in $\cK$ and let $\fkp$ be a prime of $E$ above $p$.
  We also fix a field isomorphism $\iota_\ell:\ol{\Q}_\ell\simeq \C$ throughout.
 Set $n=\dim_F V \ge 1$ and $F_0=F^{*=1}$.
 Let $\Spl_{F/F_0}$ denote the set of all rational primes $v$
 such that every prime of $F_0$ above $v$ splits in $F$.
  We further assume
\begin{enumerate}
  \item $F_0\neq \Q$,
  \item if a prime $v$ is ramified in $F$ then $v\in \Spl_{F/F_0}$,
  \item $\bG_{\Q_v}$ is quasi-split at all finite places $v$,
  \item the flat closure of the generic fibre in the integral models of \cite{ScholzeDefSpaces} is proper.
\end{enumerate}
Condition (4) is necessary for our machinery. A sufficient condition for (4) is that the pairing $(~,~)$ on $V$ is positive or negative definite with respect to some
complex embedding $F\hra \C$, as proved by Kai-Wen Lan, \cite{Lan}, Theorem 5.3.3.1, Remark 5.3.3.2,
cf. also \cite{ScholzeDefSpaces}, Theorem 5.8, for explanation.

Conditions (1) and (2) are imposed only to avoid issues with $L$-packets for unitary groups and will become unnecessary in due course.
We remark that this situation is somewhat more general than in \cite{Shin11}, as we allow general signatures at $\infty$.

Condition (3) contrasts the case considered in the previous section, and means that
endoscopy is seen in its full strength. It implies that $\bG(\A)$ admits an automorphic representation whose base change is $\Pi$,
for any $\Pi$ as below. (If $\bG_{\Q_v}$ is not quasi-split at a finite $v$ then there would be a local obstruction at $v$.)

Let $\Pi=\psi\otimes \Pi^1$ be an automorphic representation of $\bG(\A_\cK)\simeq \GL_1(\A_\cK)\times \GL_n(\A_F)$.
Let $\theta$ denote the automorphism of $\bG(\A_\cK)$ induced by the nontrivial
element $c$ of $\Gal(\cK/\Q)$.
Suppose that
\begin{itemize}
  \item $\Pi\simeq \Pi\circ \theta$ (such a $\Pi$ is called $\theta$-stable),
  \item the $L$-parameter for $\Pi_\infty$ is the base change of
$\varphi_\xi$,
  \item $\Pi$ is ramified only at the places of $F$ above  $\Spl_{F/F_0}$.
\end{itemize}
We will put ourselves in either
 (Case ST) or (Case END) of \cite[\S6.1]{Shin11}, to be briefly
 recalled here. The first case refers to the case when $\Pi$ is cuspidal.
In the second, $\Pi$ is required to be parabolically induced from
a $\theta$-stable cuspidal automorphic representation on a maximal proper
$\theta$-stable Levi subgroup which is isomorphic to a quadratic base change of
an elliptic endoscopic group of $\bG$. Such a Levi subgroup is isomorphic to
$\GL_1(\A_\cK)\times \GL_{m_1}(\A_F)\times \GL_{m_2}(\A_F)$
for some $m_1\ge m_2\ge 1$ such that $m_1+m_2=n$. It is isomorphic to a quadratic base change
of a (necessarily unique up to isomorphism) elliptic endoscopic group $(\bH,s,\eta)$ of $\bG$ in
all cases except when $m_1$, $m_2$ and $[F^+:\Q]$ are all odd. We may and do assume that
$s^2=1$. Let $\pi_p$ be the irreducible smooth representation of $\bG(\Q_p)$
whose base change is isomorphic to $\Pi_p$. Similarly, in (Case END), we have
a representation $\pi_{\bH,p}$ of $\bH(\Q_p)$; its $\eta$-transfer is $\pi_p$.

Fix a set of representatives for $\cE_{\el}(\bG)$ once and for all as in \cite[\S3.2]{Shin11},
by choosing a Hecke character $\varpi:\A^\times_\cK/\cK^\times \ra \C^\times$
such that $\varpi|_{\A^\times/\Q^\times}$ corresponds to the quadratic character
associated with $\cK/\Q$ via class field theory, using Lemma 7.1 of \cite{Shin11}.
Define $S$ to be the finite set of all primes which are either $p$ or a prime where
any of $F$, $\Pi$ or $\varpi$ is ramified; our assumptions imply that $S\subset \Spl_{F/F_0}$.
Define a virtual representation $H^*_\xi(\Pi)$ of $\Gal(\ol{E}/E)$ by
\begin{equation}\label{e:Pi-part}
H^*_\xi(\Pi)=\sum_{\pi_f} H^*_\xi(\pi_f)
\end{equation}
where the sum runs over $\pi_f$ such that $H^*_\xi(\pi_f)\neq 0$, $\pi_f^S$ is unramified, $BC(\iota_\ell \pi_f^S)\simeq
\Pi^S$ and $BC(\iota_\ell \pi_{S})\simeq \Pi_S$; this is a finite sum. Here $BC$ denotes
the local base change map, which makes sense as $\pi_f^S$ is unramified and
$S\subset \Spl_{F/F_0}$ (\cite[\S4.2]{Shin11}).

We have the following proposition, which is not explicitly stated in \cite{Shin11}.

\begin{prop} Let $\pi_f$ be such that $H^*_\xi(\pi_f)\neq 0$, $\pi_f^S$ is unramified and $BC(\iota_\ell\pi_f^S)\simeq \Pi^S$.
Then $BC(\iota_\ell \pi_S)\simeq \Pi_S$.
\end{prop}

\begin{proof}
  One can argue with the trace formula as in the proof of Corollary 6.5 (iv) of \cite{Shin11}, which is similar to the proof
  of Theorem 6.1 of \cite{Shin11}, but relies on the usual trace formula rather than the trace formula for Igusa varieties.
  One derives formulas just as (6.29) and (6.30) of \cite{Shin11}. Then the proposition follows from the
   strong multiplicity one theorem for general linear groups and the injectivity
   of local base change on $S$ (since $S\subset \Spl_{F/F_0}$; see \cite{Shin11}, \S4.2).
   (In (Case END), $\Pi$ is not cuspidal and it is implicitly used in the argument that
   the induced representation of an irreducible unitary representation is irreducible for general linear groups,
   cf. \cite{Shin11}, (6.2), to ensure that not only the supercuspidal support of $\Pi_S$ but $\Pi_S$ itself is uniquely determined.)
\end{proof}

Set
\begin{equation}\label{e:C} C_{\bG}=(-1)^{\dim \Sh}\tau(\bG).\end{equation}
The following generalizes \cite[Thm 6.4]{Shin11} to other compact unitary Shimura varieties.

\begin{thm}\label{t:cohom-Sh} In (Case ST)
 $$H^*_\xi(\Pi)|_{W_{E_{\fkp}}}= C_{\bG}\cdot
(r_{-\mu}\circ \varphi_{\pi_p}|_{W_{E_{\fkp}}})|\cdot|_{\fkp}^{-\dim\Sh/2}.$$
In (Case END) there exists $e_{\Pi}\in \{\pm 1\}$ (independent of $p$ and $\fkp$) such that
 $$H^*_\xi(\Pi)|_{W_{E_{\fkp}}}= C_{\bG}\cdot \frac{1}{2}\left(
(r_{-\mu}\circ \varphi_{\pi_p}|_{W_{E_{\fkp}}})
+e_{\Pi}(r_{-\mu}\circ \eta\varphi_{\pi_{\bH,p}}|_{W_{E_{\fkp}}})^s\right)|\cdot|_{\fkp}^{-\dim\Sh/2}.$$
  Here $(r_{-\mu}\circ \eta\varphi_{\pi_{\bH,p}}|_{W_{E_{\fkp}}})^s$ is the virtual representation
  $V^+_{-\mu}-V^-_{-\mu}$ of $W_{E_{\fkp}}$, where
  $V_{-\mu}=V^+_{-\mu} \oplus V^-_{-\mu}$ is the decomposition into $+1$ and $-1$ eigenspaces
  for the $s$-action on $r_{-\mu}\circ \eta\varphi_{\pi_{\bH,p}}$.
\end{thm}

\begin{rem}
  Even though $\eta\varphi_{\pi_{\bH,p}}$ is equivalent to $\varphi_{\pi_p}$, we keep the expression
  $\eta\varphi_{\pi_{\bH,p}}$ to remember how $s$ acts on the underlying vector space. Note that
  the assertion in (Case END) can be restated as
\begin{equation}\label{e:alternative}
 H^*_\xi(\Pi)|_{W_{E_{\fkp}}}= \left\{\begin{array}{cc}
C_{\bG} V^+_{-\mu}|\cdot|_{\fkp}^{-\dim\mathrm{Sh}/2}, &\mbox{if}~e_\Pi=+1,\\
 C_{\bG} V^-_{-\mu}|\cdot|_{\fkp}^{-\dim\mathrm{Sh}/2}, &\mbox{if}~e_\Pi=-1.
  \end{array}
  \right.
\end{equation}
\end{rem}

\begin{proof}
We will first only prove this theorem up to sign. The sign will be determined later in Corollary \ref{cor:sign-correct}.
  We adopt the strategy of proof  as well as notation and convention
  (e.g. the choice of Haar measures, transfer factors and intertwining operators)
  from \cite[Thm 6.1]{Shin11}. Once we begin with Theorem \ref{MainTheorem2}
  (playing the role of Proposition 5.5 in that paper),
  basically the same argument works if suitable changes are made at $p$.
  As for the choice of test functions, explained at the start of proof
  of \emph{loc. cit.}, it suffices to remark that
  $f^{\bH}_{\tau,h}$ (replacing $\phi^{\vec{n}}_{\mathrm{Ig},p}$)
  is in the image of the base change transfer for each $\bH$
  since $p$ splits in $\cK$ and that
  $f^{\bH}_{\xi}$ is the same function as $\phi^{\vec{n}}_{\mathrm{Ig},\infty}$ by construction
  (when $\bH=G_{\vec{n}}$).
  Let us write $\tilde{f}^{\bH}$ (with sub or superscript) for the function on $\bG\otimes_\Q \cK$
  whose base change transfer is $f^{\bH}$. Hence our notation $f$ and $\tilde{f}$
  correspond to $\phi$ and $f$ in \cite{Shin11}, respectively.

  Let $\bG^*$ be a quasi-split inner form of $\bG$, which is a principal endoscopic group.
  In (Case ST), (6.14) of \cite{Shin11} is supplanted by
  \[\begin{aligned}
  \tr(\Pi_p(\tilde{f}^{\bG^*}_p)A^0_{\Pi_p})&= \tr \iota_l \pi_p (f^{\bG^*}_{\tau,h}) \\
  &=  \tr\left(\tau\left|(r_{-\mu}\circ \varphi_{\pi_p}|_{W_{E_{\fkp}}})|\cdot|_{\fkp}^{-\dim\Sh/2}\right.\right)\tr(h|\pi_p)\ ,
  \end{aligned}\]
  where the last equality is justified by Theorem \ref{t:conj-GLn}.
  In (Case END), the following is put in place of (6.19) of \cite{Shin11}.\footnote{There are typos in
  (6.19) and (6.20) of \cite{Shin11}: $\Pi_p$ and $\Pi_\infty$ should read $\Pi_{\bH,p}$ and
  $\Pi_{\bH,\infty}$, resp.}
  \[\begin{aligned}
  \tr(\Pi_{\bH,p}(\tilde{f}^{\bH}_p)A^0_{\Pi_{\bH,p}})&=\tr \iota_l \pi_{\bH,p} (f^{\bH}_{\tau,h})\\
  &= \tr\left(s\tau\left|(r_{-\mu}\circ\eta\varphi_{\pi_{\bH,p}}|_{W_{E_{\fkp}}})|\cdot|_{\fkp}^{-\dim\Sh/2}\right.\right)\tr(h|\pi_p)
  \end{aligned}\]
  This time the last equality comes from Corollary \ref{c:conj-GLn}. The right hand side is equal to
  $\tr (\tau|(r_{-\mu}\circ \eta\varphi_{\pi_{\bH,p}}|_{W_{E_{\fkp}}})^s |\cdot|_{\fkp}^{-\dim\Sh/2})$.
  Set $e_\Pi=e_1^{-1}e_2$.
  With these changes, the argument of \cite{Shin11} proves Theorem \ref{t:cohom-Sh} up to sign.
\end{proof}

The theorem implies the Ramanujan-Petersson conjecture for regular algebraic conjugate selfdual cuspidal automorphic representations of $\GL_n$. This was proved previously by Clozel and Caraiani, \cite{ClozelPurity}, \cite{Caraiani}. We note that Caraiani's result is stronger in that it works for all finite places.

\begin{cor} Let $M$ be a CM field, and let $\Pi$ be a cuspidal automorphic representation of $\GL_n(\A_M)$ such that $\Pi^\vee\cong \Pi\circ c$ and $\Pi_\infty$ is regular algebraic, i.e. has the same infinitesimal character as some algebraic representation of $R_{M/\Q} \GL_n$. Then $\Pi_v$ is tempered for all finite places $v$ of $M$ above rational primes $p$ above which $M$ and $\Pi$ are unramified.
\end{cor}

\begin{rem}
  In fact the proof shows that $\Pi_v$ is tempered at all ramified places as well
  unless $n$ is even and condition (3) of the theorem fails.
\end{rem}

\begin{proof} Fix the place $v$ of $M$, and let $p$ be the rational prime below $v$. We adjoin several quadratic extensions of $\Q$ which are split at $p$ to $M$ to make the following assumptions true:
\begin{itemize}
\item $F^\prime=M$ is the composite of the totally real subfield $F_0^\prime$ and an imaginary quadratic field $\cK$ in which $p$ splits,
\item $\Pi$ and $F^\prime$ are only ramified at places above $\Spl_{F^\prime/F_0^\prime}$,
\item $F_0^\prime\neq \Q$.
\end{itemize}
Note that the quadratic base change of $\Pi$ stays conjugate selfdual and $\Pi_\infty$ stays regular algebraic. If $\Pi$ is unramified at some place where the quadratic field ramifies, then it also stays cuspidal, cf. \cite{ArthurClozel}. Finally, we adjoin one more real quadratic extension of $\Q$, split at $p$, to $F_0^\prime$ and $F^\prime$, to get $F_0$ and $F$. One checks that all assumptions remain true for $F$ in place of $F^\prime$. Let $\Pi^\prime$ be the base-change of the original $\Pi$ to $F^\prime$, and let $\Pi$ be the further base-change to $F$.

Using Lemma 7.2 of \cite{Shin11}, we arrive at a situation where the assumptions of (Case ST) are satisfied. Consider a PEL datum whose associated group $\bG$ is
quasi-split at all finite places and has as signature at infinity either
  \begin{itemize}
    \item[(i)] $(1,n-1)$, $(0,n)$, ..., $(0,n)$ or
    \item[(ii)] $(1,n-1)$, $(1,n-1)$, $(0,n)$, ..., $(0,n)$.
  \end{itemize}
  In case (ii), we require that the two infinite places of $F_0$ with signature $(1,n-1)$ lie above one place of $F_0^\prime$.
  If $n$ is odd then there is no obstruction for choosing the signature. In general there is a parity
  obstruction but one can still find such a PEL datum with one of the two signatures. Note that all assumptions of this section are satisfied.

  Let $\Sh$ denote the corresponding Shimura variety. Note that the reflex field $E$ is given by $F$ in case (i), and by $F^\prime$ in case (ii).
  In case (i), the proof of \cite[Cor 6.5]{Shin11} (following the proof presented on page 207
   of \cite{HT01}), can be combined with the Weil conjecture to imply that $\Pi$ is tempered at $v$. We note that this argument works even at ramified places.
  Let us explain the argument in the more subtle case (ii). Note that we have a place $v$ of $F^\prime$ and want to prove that $\Pi^\prime_v$ is tempered. We use the previous theorem at the place $v$ of $E=F^\prime$. The representation $r_{-\mu}\circ \varphi_{\pi_p}|_{W_{F^\prime_v}}$ is given by a unitary twist of $\varphi_{\Pi^\prime_v}^{\otimes 2}$. The classification of unitary generic representations shows that $\Pi^\prime_v$ is induced from a representation
  \[
  \bigotimes_i (\pi_i |\cdot|^{a_i}\otimes \pi_i |\cdot|^{-a_i})\otimes \bigotimes_j \pi_j\ ,
  \]
  of a Levi subgroup, where $\pi_i$ are tempered representations of smaller $\GL_k$'s, and $0<a_i<\frac 12$ are real numbers. We want to prove that the first factor is trivial. Our assumptions ensure that this representation already occurs at an unramified level, for which there exists a smooth model of our Shimura variety. Thus, it follows from the Weil conjectures and the fact that $\varphi_{\Pi^\prime_v}^{\otimes 2}$ occurs in the cohomology of the Shimura variety that $4a_i\in \mathbb{Z}$, i.e. $a_i=\frac 14$ for all $i$. If the first factor is nonzero, then it follows that $\varphi_{\Pi^\prime_v}^{\otimes 2}$ contributes nontrivially to three consecutive degrees. On the other hand, the cohomology is an alternating sum, with consecutive degrees contributing different signs (which cannot cancel each other as they have different weights). This contradiction finishes the proof.
\end{proof}

\begin{cor}\label{cor:sign-correct} In both (Case ST) and (Case END), for any $\pi_f$ such that $H^*_\xi(\pi_f)\neq 0$, $\pi_f^S$ is unramified and $BC(\iota_\ell\pi_f^S)\simeq \Pi^S$, the $\pi_f$-isotypic component $H^i_\xi(\pi_f)$ is nonzero only for $i=\dim \Sh$. In particular, the sign in Theorem \ref{t:cohom-Sh} is correct.
\end{cor}

\begin{proof} This is a direct application of the Weil conjectures.
\end{proof}

\begin{rem}
  We restricted ourselves to (Case ST) and (Case END) above for simplicity, as this is enough for the
  application of the next section. It should be possible to treat the case of general $\Pi$
  once finer information becomes available about the terms in the twisted trace formula for $\bG\otimes_\Q \cK$.
 Such information would be obtained
 by carrying over the results of \cite{White} and \cite{ArthurEndoscopy} to the case of unitary similitude groups.
\end{rem}

\section{Construction of Galois representations}\label{s:Galois}

  In this final section we remark that the approach used in this paper provides a shorter proof of
  the main result of \cite{Shin11}, which was based on a generalization of
  the method of \cite{HT01}. In particular one can bypass the use of
  Igusa varieties, the Newton stratification, and the analogues of the first and second basic identities
  (\cite[Thm VI.2.9, V.5.4]{HT01}, cf. \cite[Prop 5.2, Thm 6.1]{Shin11}) of Harris and Taylor.

\begin{thm}\label{t:ExGalRepr}(\cite[Thm 1.2]{Shin11})
  Let $F$ be any CM field. Let $\Pi$ be a cuspidal automorphic representation of
  $\GL_n(\A_F)$ such that
  \begin{enumerate}
      \item $\Pi^\vee\simeq \Pi\circ c$,
      \item $\Pi_\infty$ has the same infinitesimal
      character as some irreducible algebraic representation $\Xi^\vee$
      of the restriction of scalars $R_{F/\Q} \GL_n$, and
      \item $\Xi$ is slightly regular, if $n$ is even.
  \end{enumerate}
  Then for each prime $\ell$ and an isomorphism $\iota_\ell:\ol{\Q}_\ell\simeq \C$,
  there exists a continuous semisimple representation $R_{\ell,\iota_\ell}(\Pi):
  \Gal(\ol{F}/F)\rightarrow \GL_n(\ol{\Q}_\ell)$ such that
   for any place $y$ of $F$ not dividing $\ell$, the
   Frobenius semisimple Weil-Deligne representation associated to
   $R_{\ell,\iota_\ell}(\Pi)|_{\Gal(\ol{F}_y/F_y)}$ corresponds to $\iota_\ell^{-1}\Pi_y$
   via (a suitably normalized) local Langlands correspondence, except possibly for the
   monodromy operator $N$.
\end{thm}
  See \cite{Shin11} for the meaning of (3) and
  a complete list of assertions regarding $R_{\ell,\iota_\ell}(\Pi)$.
  It is not difficult to extend the above theorem to the case where
  $\Pi^\vee\simeq \Pi\circ c$ holds up to a character twist or where
  $\Pi^\vee\simeq \Pi$ up to a character and $F$ is a totally real field.
  Moreover using the above theorem as a starting point, several mathematicians
  have strengthened it (\cite{CH}, \cite{Caraiani}, \cite{BLGGTa},
  \cite{BLGGTb}). For instance (3) has become unnecessary and a similar
  conclusion holds at $y$ dividing $\ell$.

\begin{proof}
  In \cite{Shin11}, the main theorems of section 7
  follow from Theorem 6.4 (and its corollaries). Since our Theorem \ref{t:cohom-Sh}
  reproves Theorem 6.4 (by a different method) when applied to the Shimura varieties
  of \cite{Shin11}, the main results of that article can be derived from Theorem \ref{t:cohom-Sh} by the same reasoning.
  (Our proof of Theorem \ref{t:cohom-Sh} shows that
  the sign $e_\Pi$ is equal to $e^{-1}_1e_2$ of \cite{Shin11}. Hence it can be controlled
  by the choice of archimedean parameters
  as in Lemma 7.3 of that paper.)
  Note that as in \cite{Shin11} it is enough to consider Shimura varieties
  attached to unitary similitude groups in an odd number of variables with signature
  $(1,m-1)$, $(0,m)$, ..., $(0,m)$ at infinity where $m=n$ if $n$ is odd and $m=n+1$ if $n$ is even,
  even though Theorem \ref{t:cohom-Sh} covers more general cases.
\end{proof}

Finally, let us add one remark about the determination of the monodromy operator $N$,
which is the most delicate extra assertion in \cite[Thm 1.2]{Shin11}.
One easily proves temperedness of $\Pi_v$ for all $v$, so it remains to verify
the weight-monodromy conjecture for $R_{\ell,\iota_\ell}(\Pi)$, as in the paper of Taylor and Yoshida,
\cite{TaylorYoshida}, cf. \cite{Shin11}, Section 7. The input that one needs is a description
of the cohomology of each Newton stratum. This can be achieved by using the Langlands-Kottwitz method.
First, one has to check that the arguments of \cite{ScholzeDefSpaces} compute the cohomology
of the Newton stratum corresponding to the $\sigma$-conjugacy class $b\in B(\bG)$ once
$\phi_{\tau,h}$ is replaced by $\phi_{\tau,h}\chi_b$, where $\chi_b$ is the characteristic function
of the $\sigma$-conjugacy class $b$, at least if $j$ is sufficiently large. Now one has to follow the arguments
of this paper with $\phi_{\tau,h}\chi_b$ in place of $\phi_{\tau,h}$. Note that we have already determined
the twisted orbital integrals of $\phi_{\tau,h}$, which in turn determine those of $\phi_{\tau,h}\chi_b$.
This gives the desired description of the cohomology of each Newton stratum, which may then be used to
determine $N$.

\bibliographystyle{abbrv}
\bibliography{LKMethodStabilization}

\end{document}